\documentclass{birkjour}

\usepackage{MnSymbol}

\newtheorem{theorem}{Theorem}[section]
\newtheorem{corollary}[theorem]{Corollary}
\newtheorem{lemma}[theorem]{Lemma}
\newtheorem{proposition}[theorem]{Proposition}
\theoremstyle{definition}
\newtheorem{definition}[theorem]{Definition}
\theoremstyle{remark}

\newtheorem*{example}{Example}
\numberwithin{equation}{section}

\begin{document}

%
%
%
%
%
%
%
%
%

\title[Ideals of Graph Homomorphisms]{Ideals of Graph Homomorphisms}

\author{Alexander Engstr\"om}

\address{%
Aalto University\\
Department of Mathematics\\
PO Box 11100\\
FI-00076 Aalto\\
Finland}

\email{alexander.engstrom@aalto.fi}

\author{Patrik Nor\'en}

\address{%
Aalto University\\
Department of Mathematics\\
PO Box 11100\\
FI-00076 Aalto\\
Finland}

\email{patrik.noren@aalto.fi}

\thanks{Alexander Engstr\"om gratefully acknowledges support from the Miller Institute for Basic Research in Science at UC Berkeley. Patrik Nor\'en gratefully acknowledges support from the Wallenberg foundation.}



\date\today

\begin{abstract}
In combinatorial commutative algebra and algebraic statistics many toric ideals are constructed from graphs. Keeping the categorical structure of graphs in mind we give previous results a more functorial context and generalize them by introducing the ideals of graph homomorphisms. For this new class of ideals we investigate how the topology of the graphs influence the algebraic properties. We describe explicit Gr\"obner bases for several classes, generalizing results by Hibi, Sturmfels and Sullivant. One of our main tools is the toric fiber product, and we employ results by Engstr\"om, Kahle and Sullivant. The lattice polytopes defined by our ideals include important classes in optimization theory, as the stable set polytopes. 
\end{abstract}

\maketitle

\section{Introduction}

In this paper we introduce the ideals of graph homomorphisms. They are natural generalizations of toric ideals studied in particular in combinatorial commutative algebra and algebraic statistics. The lattice polytopes associated to them are important in optimization theory, and we can derive results on graph colorings with these ideals. Many toric ideals in the literature are defined from graphs, but usually the categorical structure is lost in the translation. Defining the objects from graph homomorphisms provide functorial constructions for free, as in homological algebra.

\subsection{A short overview of the paper}
 For every pair of graphs $G$ and $H$ the graph homomorphisms from $G$ to $H$ defines a toric ideal $I_{G\rightarrow H}$. In Section~\ref{sec:idealsOfGraphHomomorphisms} we give a proper definition of ideals of graph homomorphisms $I_{G\rightarrow H}$.
We give examples and explain how they relate to previously studied toric ideals, in particular from algebraic statistics. Some basic properties are proved, with focus on how modifications of the graphs $G$ and $H$ change the ideal of graph homomorphisms from $G$ to $H$. 

The toric fiber product introduced by Sullivant~\cite{sullivant2007} and further developed by Engstr\"om, Kahle, and Sullivant~\cite{engstromKahleSullivant2011} is explained in the context of ideals of graph homomorphisms in Section~\ref{sec:glue}.
If the intersection of two graphs $G_1$ and $G_2$ is sufficiently well-behaved with regard to a target graph $H$, this allows us to
lift algebraic properties and bases of $I_{G_1 \rightarrow H}$ and $I_{G_2 \rightarrow H}$ to $I_{G_1\cup G_2 \rightarrow H} = I_{G_1 \rightarrow H} \times_{G_1 \cap G_2 \rightarrow H} I_{G_2 \rightarrow H}$. We apply this to understand $I_{G\rightarrow H}$ for $G$ in some graph classes.

In Section~\ref{sec:normal} we review results on normality, and how to use the toric fiber product of the previous section to lift normality from particular graphs to complete classes. For example, we show that if the semigroup associated to $I_{K_3 \rightarrow H}$ is normal, then so are the ones associated to $I_{G \rightarrow H}$ when $G$ is a maximal outerplanar graph.

The independent sets of a graph $G$ are indexed by the graph homomorphism from $G$ to the graph $\downspoon$ with two adjacent vertices and one loop. In Section~\ref{sec:Independent} we study the ideals $I_{G\rightarrow \downspoon}$. First we give a convenient multigrading that is crucial for later proofs. For many families of toric ideals in algebraic statistics it is known, or conjectured, that the largest degree of an element in a minimal Markov basis is even. We show by explicit constructions that also odd degrees appear for $I_{G\rightarrow \downspoon}$. Then we derive a quadratic square-free Gr\"obner basis for $I_{G\rightarrow \downspoon}$ when $G$ is a bipartite graph, and this shows that they are normal and Cohen-Macaulay.

Following Section 8, we extend our results from bipartite graphs to graphs that become bipartite after the removal of some vertex. This is much more technically challenging.
 
Our toric ideals define lattice polytopes $P_{G \rightarrow H}$. In Section~\ref{sec:polytopes} we first study how modifications of $H$ gives faces of $P_{G \rightarrow H}$. Then we show how one of the most important classes of polytopes in optimization theory, the stable set polytopes \cite{lovasz2003}, appear naturally as isomorphic to some of our polytopes.
 
In Section~\ref{sec:ASL} we show that Hibi's algebras with a straightening law from distributive lattices \cite{hibi1987} is isomorphic to a graded part of some particular ideals of graph homomorphisms. Hibi's results on normality, Cohen-Macaulayness, and Koszulness, follows right off from our much larger class.

The ideal of graph homomorphisms whose target graph is a complete graph, $I_{G \rightarrow K_n}$, is an algebraic structure on the $n$-colorings of a graph $G$. In Section~\ref{sec:Colorings} we show how the ideals can be used to give structural information about graph colorings.

But we start off by introducing some notation from toric geometry and algebraic statistics in Section~\ref{sec:toricAlgebraicStatistics}, and a short overview of the category of graphs in Section~\ref{sec:graphHomo}.

\section{Toric geometry in algebraic statistics}\label{sec:toricAlgebraicStatistics}
The toric ideals studied in this paper are closely connected to those in algebraic statistics. While the methods from any textbook on combinatorial commutative algebra, like Miller and Sturmfels~\cite{millerSturmfels2005}, is enough to parse most algebraic statements of this paper, we want to point out some notions and particularities of algebraic statistics. For a nice introduction to this area we recommend the lectures on algebraic statistics by Drton, Sturmfels and Sullivant~\cite{drtonSturmfelsSullivant2009}.

We fix a field $\mathbf{k}$ throughout the paper. Two equivalent ways to define a toric ideals are used: For a matrix $A=(a_{ij}) \in \mathbb{Z}^{kl}_{\geq 0} $ and a polynomial ring $R=\mathbf{k}[r_1,r_2, \ldots, r_l],$ the toric ideal $I_A$ is generated by the binomials $r^{\mathbf{u}}-r^{\mathbf{v}} = r_1^{u_1}r_2^{u_2} \cdots r_l^{u_l} - r_1^{v_1}r_2^{v_2} \cdots r_l^{v_l}$ for which $A\mathbf{u}=A\mathbf{v}$. Alternatively we could have defined $I_A=I_\Phi$ as the kernel of the map $\Phi$ from $R$ to $S=\mathbf{k}[s_1,s_2, \ldots, s_k]$ defined by $\Phi(r_j)=s_{1}^{a_{1j}}s_{2}^{a_{2j}}\cdots s_{k}^{a_{kj}}$. This toric ideal cuts out a toric variety denoted $X_A$ or $X_\Phi$.

For each monomial $m$ in $S$, the \emph{fiber} of $m$ is the set of monomials in $R$ mapped to $m$ by $\Phi$. For any monomials $m'$ and $m''$ in the same fiber, there is a binomial $b$ in the toric ideal $I_A$ and a monomial $n$ in $R$ such that $m'-m''= \pm bn$. If $\mathcal{B}$ is a \emph{Markov basis} (that is, a generating set) of $I_A$, then there is a sequence of monomials
\[ m'=m_1,m_2,\ldots,m_{t-1},m_{t}=m'' \]
in the fiber such that $m_i-m_{i+1}=\pm b_{i}n_{i}$ for some binomial generator $b_i$ in $\mathcal{B},$ and monomial $n_i$ in $R$, for all $1\leq i < t$. The step from $m_i$ to $m_{i+1}$ in the sequence is referred to as a \emph{Markov step} or a \emph{Markov move}. 

There is a graph structure on the fiber of the monomial $m$. This graph $\mathcal{F}_m$ has the monomials of the fiber as vertices, and they are adjacent if there is a Markov step between them. A set of binomials is a basis if and only if every fiber graph is connected. We often view the Markov steps as having a direction imposed by the basis. If a basis $\mathcal{B}$ is constructed with an order $m' \rightarrow m''$ for each binomial $m'-m'' \in \mathcal{B}$, then this imposes an order on each Markov steps, and turns the fiber graphs directed. A \emph{sink} in a directed graph has all edges directed towards it, and in a \emph{directed acyclic graph} there are no cycles with the edges directed in consecutive order. If all fiber graphs of $\mathcal{B}$ are connected, directed acyclic, and have a unique sink, then $\mathcal{B}$ is a \emph{Gr\"obner basis}.

The \emph{degree} of a basis $\mathcal{B}$ is the maximal degree of a binomial in it. The \emph{Markov width} of a toric ideal $I$, denoted $\mu(I)$, is the minimal degree of a basis $\mathcal{B}$ of $I$. The Markov width is an important invariant of $I$ and a good complexity measure. When the toric ideals are given by graphs, it is a central question how the topological and structural properties of the graphs are reflected in the Markov width~\cite{engstromKahleSullivant2011}.

\section{The category of graphs}~\label{sec:graphHomo}
The reader is invited to recall basic graph theory from Diestel \cite{diestel2010}. A \emph{loop} is an edge attached to only one vertex. Most our graphs are simple or with loops, but never with multiple or weighted edges. We get the graph $G^\circ$ from $G$ by attaching loops to all its vertices.  Although we sometimes have loops, the complement of a graph without loops, is the ordinary complement without loops. The \emph{symmetric difference} $A\Delta B$ of two sets $A$ and $B$ contains the elements that are in exactly one of $A$ and $B$. For an integer $d$ the set $[d]$ is $\{1,2,\ldots, d\}$.  The \emph{neighborhood} $N(v)$ is the set of vertices adjacent to $v$ in a graph $G$ (excluding loops), and $N(S)=(\cup_{v\in S} N(v)) \setminus S$ for any set $S$ of vertices.  The induced subgraph of $G$ on vertex set $S$ is denoted $G[S]$. A particular case of this, is for any edge $e$ of $G$, the graph $G[e]$ is the subgraph of $G$ only containing the edge $e$. The \emph{independence 
target graph} $\downspoon$ in Figure~\ref{fig:independenceTarget} will appear frequently in later parts of the paper. 
\begin{figure}
\center
\includegraphics[scale=0.5]{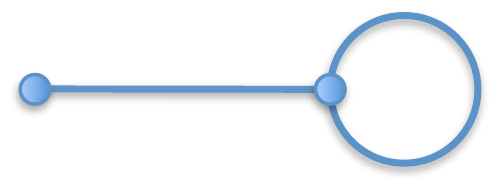}
\caption{The graph $\downspoon.$}\label{fig:independenceTarget}
\end{figure}

\begin{definition}\label{def:graphHomomorphism}
A \emph{graph homomorphism} from a graph $G$ to a graph $H$ is a function $\phi$ from the vertex set of $G$ to the vertex set of $H$ that induces a function from the edge set of $G$ to the edge set of $H$. A more formal way of stating it is that
\[ \phi:V(G)\rightarrow V(H)\]
satisfy
\[ uv\in E(G)\Rightarrow \phi(u)\phi(v)\in E(H). \]
\end{definition}
The set of graph homomorphisms from $G$ to $H$ is denoted $\mathrm{Hom}(G,H)$.  A \emph{graph isomorphism} from $G$ to $H$ is a graph homomorphism $\phi$ from $G$ to $H$ such that $\phi$ is a bijection between $V(G)$ and $V(H)$, and the inverse of $\phi$ is a graph homomorphism from $H$ to $G$. If such a map exists then $G$ and $H$ are said to be \emph{isomorphic} and can be considered to be the same.  

The following trivial facts are useful, and they capture the first aspect of why defining ideals from sets of graph homomorphisms might be a structural theory.

\begin{lemma}\label{lemma:graphHomomorphismsComposes}
Let $G_1, G_2,$ and  $G_3$ be graphs. If $\phi_1 \in \mathrm{Hom}(G_1,G_2)$ and 
$\phi_2 \in \mathrm{Hom}(G_2,G_3)$ then
$\phi_2 \circ \phi_1 \in \mathrm{Hom}(G_1,G_3)$.
\end{lemma}
\begin{proof}
The map $\phi_2 \circ \phi_1$ induces to edge sets $E(G_1) \rightarrow E(G_2) \rightarrow E(G_3)$.
\end{proof}

The graph $H_1$ is a subgraph of $H_2$, denoted $H_1 \subseteq H_2$, if $V(H_1)\subseteq V(H_2)$ and $E(H_1) \subseteq E(H_2).$

\begin{lemma}\label{lemma:expandTarget}
Let $G,H_1,H_2$ be graphs. If $V(H_1)=V(H_2)$ and $H_1 \subseteq H_2$ then $\mathrm{Hom}(G,H_1) \subseteq 
\mathrm{Hom}(G,H_2).$
\end{lemma}
\begin{proof}
Take any $\phi_1 \in \mathrm{Hom}(G,H_1)$, and set  $\phi_2(v)=v$ for $\phi_2 \in \mathrm{Hom}(H_1,H_2)$. Then use 
Lemma~\ref{lemma:graphHomomorphismsComposes}
to conclude that $\phi_1 = \phi_2 \circ \phi_1 \in \mathrm{Hom}(G,H_2)$.
\end{proof}
\begin{lemma}\label{lemma:reduceSource}
Let $G_1,G_2,H$ be graphs. If $V(G_1)=V(G_2)$ and $G_1 \subseteq G_2$ then $\mathrm{Hom}(G_1,H) \supseteq 
\mathrm{Hom}(G_2,H).$
\end{lemma}
\begin{proof}
Let $\phi_1(v)=v$ be the inclusion map in $\mathrm{Hom}(G_1,G_2)$, and take any
$\phi_2 \in  \mathrm{Hom}(G_2,H)$. Then
by Lemma~\ref{lemma:graphHomomorphismsComposes}, $\phi_2 = \phi_2 \circ \phi_1 \in \mathrm{Hom}(G_1,H)$.
\end{proof}
If $G$ is a graph and $S$ a subset of $V(G)$, then any map $\phi: V(G) \rightarrow V(H)$ can be restricted to a map
$\phi|_S : S=V(G[S]) \rightarrow V(H)$. In particular, a graph homomorphisms $\phi : G \rightarrow H$ restricts to a
graph homomorphism $\phi|_S: G[S] \rightarrow H$. For future reference we state this as a trivial lemma without proof.
\begin{lemma}\label{lemma:induceGraphHomomorphisms}
If $G,H$ are graphs, $S \subseteq V(G)$, and $\phi \in \mathrm{Hom}(G,H)$, then $ \phi|_S \in  \mathrm{Hom}(G[S],H).$
\end{lemma}

For more about graph homomorphisms, see the textbook by Hell and Ne\v set\v ril \cite{hellNesetril2005}.

\section{Ideals of Graph Homomorphisms}\label{sec:idealsOfGraphHomomorphisms}

In this section we define and prove the basic properties of the ideals of graph homomorphisms.  These ideals are toric ideals defined as kernels, and we first define rings and a map. 

\begin{definition}\label{def:ringOfGraphHomomorphisms}
For any graphs $G$ and $H$,
the \emph{ring of graph homomorphisms from $G$ to $H$} is the polynomial ring
\[ 
R_{G\rightarrow H} = \mathbf{k}\left[ r_\phi \mid  \phi : G \rightarrow H \mathrm{\ is\ a\ graph\ homomorphism} \right],
\]
and the \emph{ring of edge maps from $G$ to $H$} is the polynomial ring
\[
S_{G\rightarrow H} = \mathbf{k}\left[ s_\phi \mid e\in E(G) \mathrm{\ and\ } \phi : G[e] \rightarrow H \mathrm{\ is\ a\ graph\ homomorphism} \right].
\]
\end{definition}
For every $s_\phi$ the domain of the graph homomorphism $\phi$ is a graph consisting of one edge and its vertices. By Lemma~\ref{lemma:induceGraphHomomorphisms} the following edge separator map is well defined and a ring homomorphism.
\begin{definition}
For graphs $G$ and $H$ the \emph{edge separator map}
\[ \Phi_{G \rightarrow H} : R_{ G \rightarrow H } \rightarrow S_{ G \rightarrow H}
\]
is defined by
\[
 \Phi(r_\phi) = \prod_{e\in E(G)} s_{\phi|_e}.
\]
\end{definition}
Now we are ready to define the object of our study.
\begin{definition}
For graphs $G$ and $H$ the \emph{ideal of graph homomorphisms from $G$ to $H$}, $I_{G \rightarrow H}$, is the kernel of the
edge separator map $\Phi_{G \rightarrow H}$.
\end{definition}
The corresponding toric variety is denoted $X_{G \rightarrow H}$, and the lattice polytope $P_{G \rightarrow H}$.
\begin{example}
The hierarchical model with variables taking $n$ discrete values modeled on a graph $G$, is a common statistical model that is studied in algebraic statistics. Recall that $K_n^\circ$ is the complete graph on $n$ vertices with loops on all vertices. The hierarchical model is the special case of ideals of graph homomorphisms $I_{G \rightarrow K_n^\circ}$. Many properties and examples of these models have been studied. Develin and Sullivant~\cite{develinSullivant2003}  studied the Markov width of binary graph models and constructed Markov bases of degree four for the binary graph models when the graphs are cycles and complete bipartite graphs $K_{2,n}$.  Ho\c{s}ten and Sullivant~\cite{Serkan} gave Gr\"obner basis for the binary graph model when the graph is a cycle and found a complete facet description of the underlying polytope. Another common model in statistics is the graphical models, they are associated to the ideals $I_{G\rightarrow K_n^\circ}$ if $G$ does not contain a triangle. Hierarchical models are defined in terms of simplicial complexes, the case when the complex is the clique complex of a graph is called graphical. Geiger, Meek and Sturmfels~\cite{geiger} found conditions for when a statistical model is graphical. Kahle~\cite{Kahle} found a neighborliness property of the underlying polytope for hierarchical models. Likelihood estimation for hierarchical models is studied for example in~\cite{diaconisSturmfels1998,Dobra,Eriksson}.

\end{example}

\begin{example}
An independent set of a graph is a set of non-adjacent vertices. Another name for independents sets are stable sets, and most of the important graph theoretic concepts and problems can be stated as properties of them. Recall that the graph on two vertices with one edge and one loop is denoted $\downspoon$. Every independent set of a graph $G$ can be described as a graph homomorphism from $G$ into $\downspoon$ where the independent set is the pre-image of the vertex without a loop.
The ideal of graph homomorphisms $I_{G \rightarrow \downspoon}$, or the
\emph{ideal of independent sets}, is an important special case that we will return to in Section~\ref{sec:Independent}. The polytope associated
to $I_{G \rightarrow \downspoon}$ is isomorphic to the \emph{stable set polytope}, which is important in optimization theory.
\end{example}

\begin{example}
A \emph{graph coloring} is an assignment of colors to the vertices of a graph with no adjacent vertices getting the same color. The minimal number of colors, the \emph{chromatic number}, is an important invariant of a graph. An upper bound for the chromatic number can easily be achieved by giving an explicit coloring, but to prove that a certain number of colors is indeed needed, is much more difficult. There are many simplistic ways to associate algebraic structures to graphs in attacking this problem, but the only successful ones so far makes heavy use of the underlying category of graphs and their homomorphisms \cite{BabsonKozlov2007}. A graph coloring of a graph $G$ with $n$ vertices is nothing but a graph homomorphism from $G$ to the complete graph $K_n$. In Section~\ref{sec:Colorings} we will show how ideals of graph homomorphisms 
$I_{G \rightarrow K_n}$ can be used in the study of graph colorings.
\end{example}

\begin{example}\label{example:p3toP4}
Let $G$ be the path on the four vertices 1,2,3,4; and $H$ the path on the three vertices 1,2,3; as in Figure~\ref{figure:p3toP4}. The variable of the ring of graph homomorphisms $R_{P_4 \rightarrow P_3}$ corresponding to a graph homomorphism
$\phi : P_4 \rightarrow P_3$ is called $r_{ \phi(1)\phi(2)\phi(3)\phi(4) }$.
The variables are
$r_{1232},r_{1212},r_{2121},r_{2123},r_{2321},r_{2323},$ $r_{3212},r_{3232},$
and the ideal of graph homomorphisms $I_{P_4 \rightarrow P_3}$ is generated by
$r_{1212}r_{3232}-r_{1232}r_{3212}$ and  $r_{2121}r_{2323}-r_{2123}r_{2321}.$
\end{example}
\begin{figure}
\center
\includegraphics[scale=0.5]{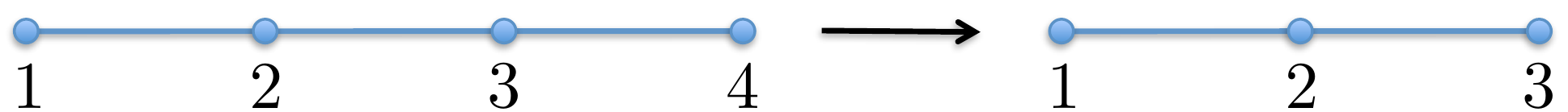}
\caption{The domain $P_4$ and target $P_3$ of the graph homomorphisms defining $I_{P_4\rightarrow P_3}$.}\label{figure:p3toP4}
\end{figure}

If $G$ have isolated vertices then $I_{G \rightarrow H}$ contains lots of uninteresting quadratic binomials. Most graphs we study lack isolated vertices, but we don't restrict to that case. If you change the source or target for a set of graph homomorphisms, it is also reflected in their rings.

\begin{lemma}\label{lemma:ringExpandTarget}
Let $G,H_1,H_2$ be graphs. If $V(H_1)=V(H_2)$ and $H_1 \subseteq H_2$ then $R_{G \rightarrow H_1} \subseteq R_{G \rightarrow H_2}$.
\end{lemma}
\begin{proof}
Use Lemma~\ref{lemma:expandTarget} and Definition~\ref{def:ringOfGraphHomomorphisms}.
\end{proof}

\begin{lemma}\label{lemma:ringReduceSource}
Let $G_1,G_2,H$ be graphs. If $V(G_1)=V(G_2)$ and $G_1 \subseteq G_2$ then $R_{G_1 \rightarrow H} \supseteq R_{G_2 \rightarrow H}$.
\end{lemma}
\begin{proof}
Use Lemma~\ref{lemma:reduceSource} and Definition~\ref{def:ringOfGraphHomomorphisms}.
\end{proof}

Ordinarily we don't want to expand the target of our graph homomorphisms as in Lemma~\ref{lemma:ringExpandTarget}, but to move around in subrings where the target is reduced. This is handled in Lemma~\ref{lemma:ringShrinkTarget}.

\begin{lemma}\label{lemma:ringShrinkTarget}
Let $G,H_1,H_2$ be graphs with $V(H_1)=V(H_2)$ and $H_1 \subseteq H_2$; and let $m,n$ be monomials in $R_{G \rightarrow H_2}$. If $m-n \in I_{G \rightarrow H_2}$ then either both  $m,n \in R_{G \rightarrow H_1}$ or both $m,n \not \in R_{G \rightarrow H_1}$.
\end{lemma}
\begin{proof}
By symmetry of $m$ and $n$, we only have to prove that if $m\not\in R_{G \rightarrow H_1}$ then $n\not\in R_{G \rightarrow H_1}$.
Assume that $m\not\in R_{G \rightarrow H_1}$ since for some graph homomorphism 
$\phi \in \mathrm{Hom}(G,H_1) \setminus \mathrm{Hom}(G,H_2)$ the variable $r_\phi$ divides $m$. This is certified by an edge
$e$ of $G$ mapped to $\phi(e)$ in $H_2\setminus H_1$. The images of $n$ and $m$ under $\Phi_{G \rightarrow H_2}$ are the same,
so there is a graph homomorphism $\phi' : G \rightarrow H_2$ such that $r_{\phi'}$ divides $n,$ and $\phi'$ sends $e$ to 
$\phi(e) \in H_2\setminus H_1$. This shows that $r_{\phi'}$ and hence $n$ is not in $R_{G \rightarrow H_1}$.
\end{proof}

\begin{theorem}\label{theorem:restrictMarkovBases}
Let $G,H_1,H_2$ be graphs with $V(H_1)=V(H_2)$ and $H_1 \subseteq H_2$. 
If $\mathcal{B}$ is a basis of $I_{G\rightarrow H_2}$, then $\mathcal{B} \cap I_{G \rightarrow H_1}$ is a basis of 
$I_{G\rightarrow H_1}$.
\end{theorem}
\begin{proof}
If $m$ and $m'$ are monomials in $R_{G \rightarrow H_2 }$ and $m-m' \in I_{G \rightarrow H_1}$, then there are monomials
$m=m_0,m_1, \ldots, m_k = m'$ such that
\[ m_1-m_0,\,  m_2-m_1,\, \ldots, m_k-m_{k-1} \in I_{G \rightarrow H_2 }, \]
and each binomial $m_i - m_{i-1}$ equals some $m_i'b_i$ where $m_i$ is a monomial in $R_{G \rightarrow H_2 }$ and $b_i$
is a binomial in $\mathcal{B}$. We want to show that each $b_i$ is in $\mathcal{B} \cap  I_{G \rightarrow H_1}$ to prove
that $I_{G \rightarrow H_1} =  \langle \mathcal{B} \cap I_{G \rightarrow H_1} \rangle$. To do this we find that each $m_i'b_i$ is in $R_{G \rightarrow H_1}$.

We assumed that $m-m'\in I_{G \rightarrow H_1}$, so in particular $m=m_0 \in R_{G \rightarrow H_1}$. By 
Lemma~\ref{lemma:ringShrinkTarget} then also $m_1\in  R_{G \rightarrow H_1}$. Repeating the same argument, gives that all $m_i$, and there differences $m_i'b_i$, are in $R_{G \rightarrow H_1}$.
\end{proof}

\begin{corollary}\label{corollary:restrictMarkovBases}
If $G,H_1,H_2$ are graphs with $V(H_1)=V(H_2)$ and $H_1 \subseteq H_2$ then the Markov widths are related by
\[ \mu\left( I_{G \rightarrow H_1} \right) \leq \mu \left( I_{G \rightarrow H_2} \right). \]
\end{corollary}
\begin{proof}
Let  $\mathcal{B}$ be a degree $\mu( I_{G \rightarrow H_2})$ basis of $I_{G \rightarrow H_2}$. Restricting $\mathcal{B}$
to $I_{G \rightarrow H_1}$ gives a basis according to Theorem~\ref{theorem:restrictMarkovBases}, and that one is at most
of the same degree as $\mathcal{B}$.
\end{proof}

\section{Gluing together graphs}\label{sec:glue}

In structural graph theory it is studied how graph classes either can be defined by forbidden minors, or by being glued together from simple starting graphs~\cite{robertsonSeymour2004}. In algebraic statistics, when ideals are formed from graphs, one can ask if there is an operation on the level of ideals corresponding to gluing the graphs. The first algebraic result in this direction, collecting several scattered results and giving them a theoretical foundation, was obtained by Sullivant \cite{sullivant2007} when he defined the \emph{toric fiber product} and showed how to make use of it in the codimension zero case. In codimension one the first result was proved by Engstr\"om \cite{engstrom2011} and it was used to prove that cut ideals of $K_4$-minor free graphs are generated by quadratic square-free binomials, as conjectured by Sullivant and Sturmfels~\cite{sturmfelsSullivant2008}. The first systematic treatment of higher codimensions, with a clear connection to structural graph theory, was recently done by Engstr\"om, Kahle, and Sullivant \cite{engstromKahleSullivant2011}. In this section we use the toric fiber product to find generators of ideals of graph homomorphisms.

The integer matrix in the definition of a toric ideal can also be regarded as a configuration of integer vectors. For two vector configurations  $B=\{\mathbf{b}^i_j \mid i \in [r], j \in [s_i] \} \subset \mathbb{Z}^{d_1}_{\geq 0} $ and  $C=\{\mathbf{c}^i_j \mid i \in [r], j \in [t_i] \} \subset  \mathbb{Z}^{d_2}_{\geq 0}$ we get toric ideals in $\mathbf{k}[x_1,x_2,
\ldots, x_{d_1}]$ and  $\mathbf{k}[y_1,y_2, \ldots, y_{d_2}]$ defined by $I_B = \langle x^{\mathbf{u}} - x^{\mathbf{v}} \mid B\mathbf{u}=B\mathbf{v}   \rangle$ and $I_C = \langle y^{\mathbf{u}} - y^{\mathbf{v}} \mid C\mathbf{u}=C\mathbf{v} \rangle.$ Assume that there is a vector configuration $A=\{\mathbf{a}^1, \mathbf{a}^2, \ldots,   \mathbf{a}^r \} \subset  \mathbb{Z}^{e}_{\geq 0}$ and linear maps $\pi_l : \mathbb{Z}^{d_l} \rightarrow \mathbb{Z}^{e}$ satisfying $\pi_1(\mathbf{b}^i_j)= \mathbf{a}^i$ and $\pi_2(\mathbf{c}^i_j)= \mathbf{a}^i$ for all the vectors. Their
\emph{toric fiber product} is the toric ideal
\[ I_{B} \times_A I_C =  \langle z^{\mathbf{u}} - z^{\mathbf{v}} \mid  (B\times_A C)  \mathbf{u}= (B\times_A C) \mathbf{v}   \rangle   \]
in $\mathbf{k}[z_1,z_2, \ldots, z_{d_1+d_2}]$
where $(B\times_A C) = \{ ( \mathbf{b}^i_j, \mathbf{c}^i_k ) \mid i \in [r], j \in [s_i], k \in [t_i] \}.$

\begin{proposition} \label{prop:tfp}
Let $G_1, G_2$ and $H$ be graphs. If $G_1 \cap G_2$ is an induced subgraph of both $G_1$ and $G_2$ then
\[ I_{G_1 \rightarrow H} \times_{G_1\cap G_2 \rightarrow H} I_{G_2 \rightarrow H} = I_{G_1 \cup G_2 \rightarrow H}. \]
\end{proposition}
\begin{proof}
Let $A$ be the vector configuration defining the toric ideal $I_{G_1 \cap G_2 \rightarrow H}$. Any graph homomorphism $\phi:G_i \rightarrow H$ restricts to a graph homomorphism
$\phi | _{G_1\cap G_2} : G_1\cap G_2 \rightarrow H.$ This gives the linear $\pi$--maps from the vector configurations defining  $I_{G_1 \rightarrow H}$ and $I_{G_2 \rightarrow H}$ to $A$.
\end{proof}

When the subscript of $\times$ is clear, as it almost always is in our applications of the toric fiber product, then we drop it from the notation. The easiest toric fiber products to work with are when the vectors in $A$ are linearly independent, because then there is a procedure to get the basis of the product from the bases of the factors. We now describe this procedure of Sullivant~\cite{sullivant2007} in the context of ideals of graph homomorphism.

\begin{proposition} \label{prop:tfpMoves}
Let $\mathcal{B}_i$ be a generating set of $I_{G_i \rightarrow H}$ for $i=1,2,$ and let $A$ be the vector configuration defining $I_{G_1 \cap G_2 \rightarrow H}$. If $G_1 \cap G_2$ is an induced subgraph of both $G_1$ and $G_2$, and the vectors of $A$ are linearly independent, then
\[ Lift(\mathcal{B}_1) \cup Lift(\mathcal{B}_2) \cup Quad \]
is a generating set of  $I_{G_1 \rightarrow H} \times I_{G_2 \rightarrow H} = I_{G_1 \cup G_2 \rightarrow H}$, where $Lift(\mathcal{B}_1)$ is the set
\[ 
\left\{
\prod_{i=1}^d r_{\phi_i} - \prod_{i=1}^d r_{\phi_i'} \in R_{G_1\cup G_2 \rightarrow H} \left| 
\begin{array}{l}
\prod_{i=1}^d r_{\phi_i |_{G_1}} - \prod_{i=1}^d r_{\phi_i'|_{G_1}} \in \mathcal{B}_1 \\ 
\textrm{and }\phi_i |_{G_2}=\phi_i' |_{G_2}\textrm{ for all }i \\
\end{array}
\right.\right\},
\]
$Lift(\mathcal{B}_2)$ is the set
\[ 
\left\{
\prod_{i=1}^d r_{\phi_i} - \prod_{i=1}^d r_{\phi_i'} \in R_{G_1\cup G_2 \rightarrow H} \left| 
\begin{array}{l}
\prod_{i=1}^d r_{\phi_i |_{G_2}} - \prod_{i=1}^d r_{\phi_i'|_{G_2}} \in \mathcal{B}_2 \\
\textrm{and } \phi_i |_{G_1}=\phi_i' |_{G_1}\textrm{ for all }i
\end{array}
\right.\right\},
\]
and $Quad$ is the set
\[
\left\{ 0 \neq
r_{\phi_1}r_{\phi_2}-r_{\phi_3}r_{\phi_4} \in R_{G_1\cup G_2 \rightarrow H}
\left|
\begin{array}{ll}
\phi_1|_{G_1}= \phi_3|_{G_1}, &  \phi_1|_{G_2}= \phi_4|_{G_2}, \\
\phi_2|_{G_1}= \phi_4|_{G_1}, &  \phi_2|_{G_2}= \phi_3|_{G_2}. \\
\end{array}
\right.
\right\}.
\]
\end{proposition}

\begin{proof}
This is a direct application of Corollary 14 in \cite{sullivant2007}, where in a general context Quad is defined in Proposition 10 and Lift is defined in Definition 11. This setup is also discussed in \cite{engstromKahleSullivant2011} in a more general context.
\end{proof}

Our main use of the previous proposition is a natural extension of the similar results for hierarchical models.

\begin{lemma}\label{lemma:glue}
Let $G_1$ and $G_2$ be two graphs whose intersection is one of $ \emptyset, K_1,$ $K_1^\circ,$ or $K_2$; and let $H$ be a graph. Then 
\[\mu(I_{G_1 \cup G_2  \rightarrow H}) \leq \max ( 2, \mu( I_{G_1 \rightarrow H}) , \mu( I_{G_2 \rightarrow H} ) ).\]
\end{lemma}
\begin{proof}
By Proposition~\ref{prop:tfp} the ideal of graph homomorphisms $I_{G_1 \cup G_2  \rightarrow H}$ is the toric fiber product $I_{G_1 \rightarrow H} \times I_{G_2 \rightarrow H}$. The vector configurations defining the toric ideals $I_{G_1 \cap G_2  \rightarrow H}$ are linearly independent if $G_1 \cap G_2$ is one of $ \emptyset, K_1, K_1^\circ, K_2$. We apply Proposition~\ref{prop:tfpMoves} to bound the Markov width. Let $\mathcal{B}_i$ be a generating set of $I_{G_i \rightarrow H}$ with binomials of degree at most $\mu(  I_{G_i \rightarrow H})$ for $i=1,2.$ By construction in Proposition~\ref{prop:tfpMoves} the binomials in $Lift(\mathcal{B}_i)$ are of degree at most $\mu(  I_{G_i \rightarrow H})$, 
the binomials in Quad are quadrics, and hence $\mu(I_{G_1 \cup G_2  \rightarrow H}) \leq \max ( 2, \mu( I_{G_1 \rightarrow H}) , \mu( I_{G_2 \rightarrow H} ) )$ since $ Lift(\mathcal{B}_1) \cup Lift(\mathcal{B}_2) \cup Quad $ generates $I_{G_1 \cup G_2  \rightarrow H}$.
\end{proof}

\begin{theorem}\label{theo:forest}
If $G$ is a forest then $I_{G\rightarrow H}$ is generated by square-free quadratic binomials.
\end{theorem}
\begin{proof}
If $G$ is a vertex this is true. We defer the case of that $G$ has several components to the end and assume that $G$ is a tree. The proof is by induction on the number of edges. If $G$ is an edge then $I_{G\rightarrow H}$ is trivial.  Otherwise cover $G$ by two trees $G_1$ and $G_2$ that both have at least one edge such that they intersect in a vertex. By induction both $I_{G_1\rightarrow H}$ and $I_{G_2\rightarrow H}$ are generated by quadrics, and then so is their union by Lemma~\ref{lemma:glue}. That they are square-free follows from that square-free binomials lifts to square-free, and that all binomials from Quad are square-free, in Proposition~\ref{prop:tfpMoves}

If $G$ is not a tree but a forest, then the same argument but gluing over empty sets apply.
\end{proof}

An \emph{outerplanar graph} is a graph that can be drawn in the plane with straight edges and with its vertices on a circle without any edges crossing each other. A maximal outerplanar graph is thus a triangulation of an $n$--gon.

\begin{theorem}\label{theo:outerplanar}
If $G$ is a maximal outerplanar graph on at least three vertices, then $\mu(I_{G\rightarrow H}) = \max(2, \mu(I_{K_3\rightarrow H})).$
\end{theorem}
\begin{proof}
The proof is by induction on the number of triangles in $G$. The statement is clearly true if $G$ is a triangle. If $G$ has more than one triangle, then there is a way to decompose $G$ into graphs $G_1$ and $G_2$ such that both of them are maximal outerplanar graphs with at least one triangle, and their intersection is an edge. By an application of 
Lemma~\ref{lemma:glue} we are done.
\end{proof}

\begin{example}\label{ex:K3K4}
The ideal of graph homomorphisms of four-colorings of a maximal outerplanar graph $G$, $I_{G\rightarrow K_4}$, is generated by binomials of degree 2 and 12. To see why this is true we not only need Theorem~\ref{theo:outerplanar}, but also the explicit description in Proposition~\ref{prop:tfpMoves}.
Using the 4ti2 software~\cite{42}  we computed that the toric ideal $I_{K_3 \rightarrow K_4}$ is generated by the degree 12 binomial
 \[ 
 \begin{array}{l}
r_{123}r_{214}r_{341}r_{432}r_{231}r_{142}r_{413}r_{324}r_{312}r_{421}r_{134}r_{243} - \\
r_{124}r_{213}r_{342}r_{431}r_{234}r_{143}r_{412}r_{321}r_{314}r_{423}r_{132}r_{241}.
\end{array}
 \]
 The binomial can be described using a permutation representation of the alternating group on four elements. When we glue together two maximal outerplanar graphs, any binomial of degree 12 will lift to a binomial of degree 12. The quadratics will lift to quadratics, and the Quad moves will only give quadratics.
\end{example}

Propostion~\ref{prop:tfpMoves} is a Corollary of a Theorem about Gr\"obner bases by Sullivant~\cite{sullivant2007}. For future reference we state this theorem in the special case of ideals of graph homomorphisms.  There is another useful type of partial order on monomials called a weight order: Let $\omega=(\omega_1,\ldots,\omega_d)$ be a vector of weights. The \emph{weight order} $<_\omega$ on the monomials in the variables $x_1,\ldots,x_n$ is defined by $x^{a_1}_1\cdots x^{a_d}_d<_\omega x^{b_1}_1\cdots x^{b_d}_d$ if $\omega_1a_1+\cdots+\omega_d a_d<\omega_1b_1+\cdots+\omega_d b_d$. A Gr\"obner basis $\mathcal{B}$ of an ideal $I$ with respect to a weight order is a finite generating set of $I$ with the property that the initial monomials of $\mathcal{B}$ generate the initial ideal of $I$.

Let $\Phi$ be a homomorphism between polynomial rings such that $\phi$ sends each variable to a monomial. A weight vector for the image of $\omega$ induces weight vector $\Phi^*\omega$ on the domain such that the weight of a monomial is the weight of the image of the monomial.

Let $G_1$ and $G_2$ be graphs such that their edge sets agree on their intersection and let $G=G_1\cup G_2$. 
Define a ring homomorphism $\Phi_{(G_1,G_2) \rightarrow H}$ from $\mathbf{k}\left[ r_\phi \mid \phi\in \mathrm{Hom}(G,H) \right] $ to $ \mathbf{k}\left[ r_\phi \mid  \phi \in \mathrm{Hom}(G_1,H)\cup \mathrm{Hom}(G_2,H)\right] $ by 
\[\Phi_{(G_1,G_2) \rightarrow H}(r_\phi)=r_{\phi\mid_{G_1}}r_{\phi\mid_{G_2}}.\]

\begin{proposition}
Let $\mathcal{B}_i$ be a Gr\"obner basis of $I_{G_i \rightarrow H}$ with respect to $\omega_i$ for $i=1,2;$ and let $A$ be the vector configuration defining $I_{G_1 \cap G_2 \rightarrow H}$. Assume that $ Quad$ is a Gr\"obner basis with respect to $\omega$. If the vectors of $A$ are linearly independent, then
\[ Lift(\mathcal{B}_1) \cup Lift(\mathcal{B}_2) \cup Quad \]
is a Gr\"obner basis of  $ I_{G_1 \cup G_2 \rightarrow H}$ with respect to $\Phi_{(G_1,G_2) \rightarrow H}^*\omega+\epsilon \omega$ for sufficiently small $\epsilon>0$.
\end{proposition}
\begin{proof}
This is Theorem 13 in \cite{sullivant2007} applied to ideals of graph homomorphisms.
\end{proof}

\section{Normality and related algebraic properties}\label{sec:normal}

In this section we very briefly survey some of the typical algebraical properties that are consequences of a good combinatorial understanding of generating sets of toric ideals. For more discussions of these topics we refer to Fr\"oberg for Koszul algebras~\cite{froberg1999},
Hochster for normal semigroups~\cite{hochster1972}, and Bruns and Herzog for Cohen-Macaulay rings~\cite{brunsHerzog1998}.

If $I$ is a toric ideal in a polynomial ring $R$ over a field $\mathbf{k}$, then $R/I$ is isomorphic to a semigroup ring $\mathbf{k}[B]$ where $B$ is a semigroup~\cite{coxLittleShea2007}. In Chapter 13 of Sturmfels textbook on Gr\"obner bases and polytopes~\cite{sturmfels1996} it is proved that if a toric ideal has a square-free Gr\"obner basis, then its associated semigroup is normal. It is a theorem of Hochster~\cite{hochster1972} that if 
$I$ is a homogenous toric ideal in $R$ whose associated semigroup is normal, then $R/I$ is Cohen-Macaulay. The last two statements are usually bundled up:
\begin{proposition}\label{prop:squarefreeNormalCM}
If a homogenous toric ideal $I$ in $R$ has a squarefree Gr\"obner basis, then its associated semigroup is normal, and $R/I$ is Cohen-Macaulay.
\end{proposition}
The following proposition was proved by Anick~\cite{anick1986}.
\begin{proposition}\label{prop:quadraticKoszul}
If $I$ is an ideal with a quadratic Gr\"obner basis in a ring $R$, then $R/I$ is Koszul.
\end{proposition}

Many results about normality in algebraic statistics can be derived from the results of Section 5 in a paper by Engstr\"om, Kahle, and Sullivant~\cite{engstromKahleSullivant2011}. We will now explain that method in the context of ideals of graph homomorphisms using the toric fiber product described in the previous section of this paper.
\begin{lemma}\label{lemma:normalGlue}
Let $I_{G_i \rightarrow H}$ for $i=1,2,$ be ideals whose semigroups are normal,  and let $A$ be the vector configuration defining $I_{G_1 \cap G_2 \rightarrow H}$. If $G_1 \cap G_2$ is an induced subgraph of both $G_1$ and $G_2$, and the vectors of $A$ are linearly independent, then the semigroup associated to $I_{G_1 \rightarrow H} \times I_{G_2 \rightarrow H} = I_{G_1 \cup G_2 \rightarrow H}$ is normal.
\end{lemma}
Using this lemma we can proceed as in Theorem~\ref{theo:outerplanar} to lift results from small graphs to complete classes.
\begin{proposition}
Let $H$ be a graph with $I_{K_3 \rightarrow H}$ normal, then for every maximal outerplanar graph $G$, the ideal $I_{G \rightarrow H}$ is normal.
\end{proposition}
\begin{proof}
Use the same recursive gluing procedure as in the proof of Theorem~\ref{theo:outerplanar} and lift the property of normality in each step by 
Lemma~\ref{lemma:normalGlue}.
\end{proof}
In the same spirit, but using the proof of Theorem~\ref{theo:forest} as a template, one can see that $I_{G \rightarrow H}$ is normal whenever $G$ is a forest. On the other hand, by an easy slight sharpening of Theorem~\ref{theo:forest}, we know that these ideals have quadratic square-free Gr\"obner bases, and are normal and Cohen-Macaulay by Proposition~\ref{prop:squarefreeNormalCM}.

\section{Ideals of graph homomorphisms from independent sets}\label{sec:Independent}

In this section we study ideals of graph homomorphisms from independent sets. An independent set of a graph $G$ can be represented as a graph homomorphism from $G$ into the graph $\downspoon$ by sending all vertices of the independent set onto the vertex without the loop, and the other ones onto the looped vertex. The indeterminate $r_{\phi} \in R_{G\rightarrow \downspoon}$ representing the independent set $S$ of $G$ is denoted $r_S$.

We now introduce a $\mathbb{Z}^{V(G)}$ multigrading $d$ on $R_{G\rightarrow \downspoon}$ by
\[ d_v(r_S) = \left\{ 
\begin{array}{cl}
1 & v\in S, \\
0 & v\not\in S; \\
\end{array}
\right. \] 
for any vertex $v$ of $G$. This extends to any monomial $m=r_{S_1}\cdots r_{S_n}$ by $d_v(m)=d_v(r_{S_1}) + \cdots + d_v(r_{S_n})$.
To determine the kernel of the map $\Phi_{G\rightarrow \downspoon}$ we only need the multigrading $d$ according to this lemma.
\begin{lemma}\label{lemma:multigrading}
Let $G$ be a graph and let $m$ and $n$ be monomials in $R_{G \rightarrow \downspoon}$ of the same degree. Then the binomial
$m-n$ is in $I_{G \rightarrow \downspoon}$ if and only if $d_v(m)=d_v(n)$ for all vertices $v$ of $G$.
\end{lemma}
\begin{proof}
That the multidegrees of $m$ and $n$ are equal when their difference is in the kernel is clear, and the proof amounts to showing the other direction.
Stated otherwise, we want to show that $\Phi_{G \rightarrow \downspoon}(m)$ can be uniquely determined from the multidegree of $m$.

Assume that the total degree of $m$ is $d$. An edge $e=uv$ can be sent by a graph homomorphism from $G$ to $\downspoon$ in three ways: (1) onto the straight edge with $u$ landing on the unlooped vertex, (2) onto the straight edge with $v$ landing on the unlooped vertex, and (3) onto the loop. But this is counted by the multidegree. The (1) case occurs $d_u(m)$ times, the (2) case occurs $d_v(m)$ times, and the (3) case occurs $d-d_u(m)-d_v(m)$ times. From this  $\Phi_{G \rightarrow \downspoon}(m)$ is uniquely determined.
\end{proof}

Using Lemma \ref{lemma:multigrading} it is often easier to argue about the independent sets and the multiset of vertices than about the monomials. Another way of stating the lemma above, is that the difference of two monomials is in the ideal if and only if they give the same multiset of vertices.

\subsection{The top graded part}

There is another natural grading on the monomials in $R_{G\rightarrow \downspoon}:$ by the number of vertices in the independent sets. This grading is important since it cuts out ideals that are previously studied. The \emph{independence number} $\alpha(G)$ of a graph $G$ is the size of the largest independent set of $G$. Alternatively, $\alpha(G)$ could have been defined as the smallest number satisfying
$ 0 \leq \sum_{v \in V(G)} d_v(r_S) \leq \alpha(G) $
for all $r_S$ in $R_{G\rightarrow \downspoon}.$ One consequence of this inequality, is that if
$ r_{S_1}r_{S_2} \cdots r_{S_d} - r_{T_1}r_{T_2} \cdots r_{T_d} \in I_{G\rightarrow \downspoon} $
and 
$ \sum_{v \in V(G)} d_v(r_{S_i}) = \alpha(G) $
for all $1\leq i \leq d,$ then 
$ \sum_{v \in V(G)} d_v(r_{T_i}) = \alpha(G) $
for all $1\leq i \leq d.$ This shows that the following definition makes sense.
\begin{definition}
The \emph{top graded part} of $R_{G\rightarrow \downspoon}$ is
\[ R_{G\rightarrow \downspoon}^{\mathtt{top}} 
 = \mathbf{k}\left[ r_S \in R_{G\rightarrow \downspoon}  :  \sum_{v \in V(G)} d_v(r_{S}) = \alpha(G)   \right],
\]
and the \emph{top graded part} of $I_{G\rightarrow \downspoon}$ is $I_{G\rightarrow \downspoon}^{\mathtt{top}}= I_{G\rightarrow \downspoon} \cap R_{G\rightarrow \downspoon}^{\mathtt{top}}. $
\end{definition}

A toric ideal can be defined in terms of a polytope. This polytope is studied in section~\ref{sec:polytopes} but we note here that the top graded part correspond to a face of this polytope. The top graded part of the toric ideal associated to the independent sets of a graph correspond to a face of the polytope.

\subsection{Any Markov width is possible}
For many toric ideals in algebraic statistics it seems that only even Markov widths are allowed \cite{MBD}. But this is not the case for ideals of graph homomorphisms from independent sets. 

We have performed computations on the ideals $I_{G\rightarrow \downspoon}$ for graphs $G$ with few vertices. Of all connected graphs with no loops, and eight or fewer vertices, there are $439$ with $\mu(I_{G\rightarrow \downspoon})=3$ and only four with $\mu(I_{G\rightarrow \downspoon})=4$. All the complete graphs have $\mu(I_{K_n\rightarrow \downspoon})=0$ and the rest have $\mu(I_{G\rightarrow \downspoon})=2$. The graphs with $\mu(I_{G\rightarrow \downspoon})=4$ are the graphs with eight vertices depicted in Figure \ref{g2}. The Markov width is low for all graphs with few vertices, but it does grow and we construct graphs $G$ with $\mu(I_{G\rightarrow \downspoon})=k$ for any integer $k\geq 2$ in Theorem \ref{t7}.

\begin{example}\label{ex:k2k3}
The smallest graph with a Markov width larger than two is $K_2\times K_3$, the skeleton of a tent. It is two cycles of length $3$ where a vertex in one of the cycles is connected with the corresponding vertex in the other cycle. It is drawn in Figure \ref{g1}. It has a basis containing one element of degree $3$,
$r_{15}r_{26}r_{34}-r_{16}r_{24}r_{35},$ and it has the quadratic elements $r_{15}r_{\emptyset}-r_1r_3,r_{16}r_{\emptyset}-r_1r_6,r_{24}r_{\emptyset}-r_2r_4,r_{26}r_{\emptyset}-r_2r_6,r_{34}r_{\emptyset}-r_3r_4,r_{35}r_{\emptyset}-r_3r_5$.
\end{example}
The graph in the previous example is a special case of a type with arbitrary large Markov width. The next of this type of graph is one of the four on at most eight vertices with Markov width four. It is the complement of a cycle $C_8$, and it is drawn in Figure \ref{g2}.

\begin{figure}
\center
\includegraphics[scale=0.5]{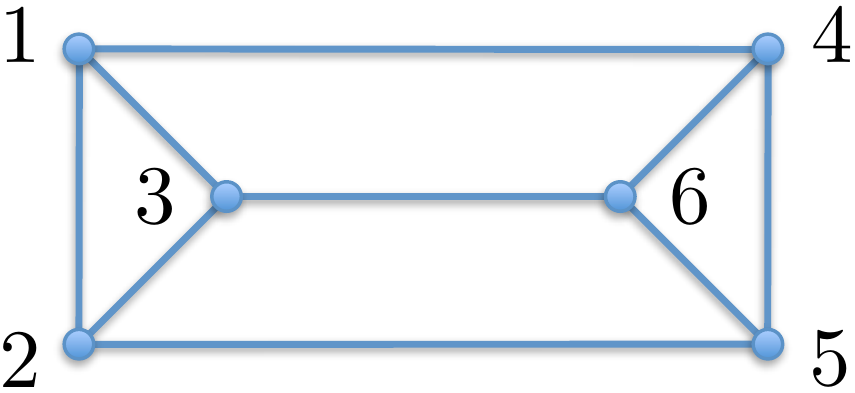}
\caption{The smallest graph with a Markov width larger than two.}\label{g1}
\end{figure}

\begin{figure}
\center
\includegraphics[width=\textwidth]{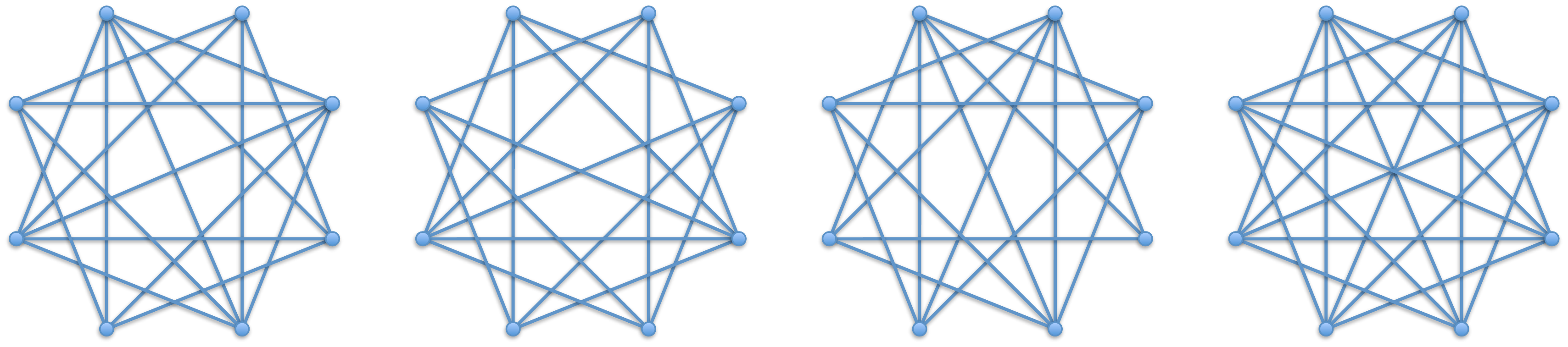}
\caption{The graphs with at most eight vertices and Markov width four. The rightmost one is the complement of $C_8$.}\label{g2}
\end{figure}
\begin{theorem}\label{t7}
If $k\ge 2$ then $\mu(I_{\overline{C_{2k}}\rightarrow \downspoon})=k.$
\end{theorem}
\begin{proof}
Consider the cycle $C_{2k}$ with vertices $0,1, \ldots 2k-1$ and edges $\{v,v+1\}$ counting modulo $2k$. 
We prove that the complement $\overline{C_{2k}}$ of $C_{2k}$ satisfies $\mu(I_{\overline{C_{2k}}\rightarrow \downspoon})=k$.
Let $b$ be the degree $k$ binomial 
\[r_{\{0,1\}}r_{\{2,3\}}\cdots r_{\{2k-2,2k-1\}}-r_{\{1,2\}}r_{\{3,4\}}\cdots r_{\{2k-1,0\}}.\] 
Both of the monomials in $b$ has multidegree one for every vertex of $\overline{C_{2k}}$, and $b \in I_{\overline{C_{2k}}\rightarrow \downspoon}$ by Lemma~\ref{lemma:multigrading}. The binomial $b$ and the quadrics of $I_{\overline{C_{2k}}\rightarrow \downspoon}$ will form a basis of it.

Say that $m$ and $n$ are monomials  and $m-n \in I_{\overline{C_{2k}}\rightarrow \downspoon}$. We should prove that 
$m$ and $n$ can reach each other by Markov moves. The proof is by induction on the degree of $m$. If the degree is two, then by construction of the basis we are done. If the degree of $m$ is larger than two, we find Markov moves from $m$ to $m'$ such that $m'$ and $n$ have a common factor, and then we are done by induction on the degree.

So,  let $m$ and $n$ be monomials with no common factors. There are two cases:
\begin{itemize}
\item[1.] \emph{The monomial $m$ (or by symmetry $n$) contains a factor $r_{\{v\}},$ where $v$ is a vertex of $\overline{C_{2k}}$.}

The monomial $n$ contains  $r_{\{v,v+1\}}$ or $r_{\{v-1,v\}}$, and without loss of generality we assume the first mentioned. It follows that $m$ contains $r_{\{v+1\}}$ or $r_{\{v+1,v+2\}}$. If $m$ contains $r_{\{v+1\}}$ then the Markov move from $r_{\{v\}}r_{\{v+1\}}$ to $r_{\{v,v+1\}}r_\emptyset$ introduce the common factor $r_{\{v,v+1\}}$. Otherwise $m$ contains $r_{\{v+1,v+2\}}$ and the Markov move from $r_{\{v\}}r_{\{v+1,v+2\}}$ to $r_{\{v,v+1\}}r_{\{v+2\}}$ introduce the same common factor.

\item[2.] \emph{There are no factors $r_{\{v\}}$ in $m$ or $n$.}

If $m$ contains $r_{\{v,v+1\}}$ then $n$ contains $r_{\{v+1,v+2\}}$. And then $m$ contains $r_{\{v+1,v+2\}}$ because of that. Proceeding around the cycle we get that $m$ contains one of the monomials in $b,$ and $n$ contains the other one. The Markov move using $b$ introduces $k$ common variables.
\end{itemize}
\end{proof}

In the next section we show that if $G$ is bipartite then $\mu(I_{G\rightarrow \downspoon})\leq 2$, and that this is also true if $G$ becomes bipartite after removing a vertex. For some 3-partite graphs $\mu(I_{G\rightarrow \downspoon})\leq 2$, but $\mu(I_{\overline{C_{6}}\rightarrow \downspoon})=3$ according to Theorem \ref{t7}. We demonstrated the existence of a graph with $\mu(I_{G\rightarrow \downspoon})\leq k$ by a $k$-partite graph, and one could speculate that many parts are forced. It turns out that this is not the case,
but it is unclear if $\mu(I_{G\rightarrow \downspoon})$ is limited for 3-partite graphs.

\begin{theorem} For any graph $G$ there is a 4-partite graph $G'$ satisfying 
\[
\mu(I_{G\rightarrow \downspoon})\leq \mu(I_{G'\rightarrow \downspoon}).
\]
\end{theorem}

\begin{proof}
We construct $G'$ from $G$ by transforming the edges of $G$. For each edge $uv$ of $G$ introduce three new vertices $w_{uv}, w_{u^\ast v},$ and $w_{uv^\ast}.$ Remove the edge $uv$ and add the edges
\[
uw_{uv}, uw_{u^\ast v}, w_{uv}w_{u^\ast v}, w_{uv}w_{uv^\ast}, w_{u^\ast v}w_{uv^\ast}, vw_{uv}, vw_{uv^\ast}.
\]
The graph $G'$ is 4-partite with the vertices of $G$ in one part and the other three parts comes from a blown-up triangle attached to $G$ in a particular way.

For each independent set $I$ of $G$ we construct a maximal independent set $I'$ of $G'$ like this: Keep all of the independent vertices of $G$ in $G'$, and for each edge $uv$ of $G$:
\begin{itemize}
\item[(1)] if neither $u$ nor $v$ is in $I$, then add $w_{uv}$ to $I'$;
\item[(2)] if $u$ is in $I$, then add $w_{uv^\ast}$ to $I'$;
\item[(3)] if $v$ is in $I$, then add $w_{u^\ast v}$ to $I'.$
\end{itemize}
Each indeterminate of $R_{G'\rightarrow \downspoon}$ corresponds to an independent set of $G'$.

Let $m-n\in I_{G'\rightarrow \downspoon}$ where all variables in $m$ correspond to independent sets in $G$ as above. Let $uv$ be an edge in $G$ and consider the graph induced by $u,v,w_{uv}, w_{u^\ast v},$ and $w_{uv^\ast}$, all the independent sets coming from $m$ are maximal in this graphs and always contain one of $w_{uv}, w_{u^\ast v},$ and $w_{uv^\ast}$.  The variables in $n$ also have this property, otherwise there would be some vertex $w$ in $G'$ where the degrees $d_w$ would be different in $m$ and $n$.  This implies that the variables in $n$ also correspond to independent sets in $G$. Hence any generating set of $I_{G'\rightarrow \downspoon}$ will contain a generating set of  $I_{G\rightarrow \downspoon}$ showing that $\mu(I_{G\rightarrow \downspoon})\leq \mu(I_{G'\rightarrow \downspoon}).$
\end{proof}

\subsection{Independent sets from bipartite graphs}

We will now prove a theorem that is used to describe a generating set for a bipartite graph. This will later be expanded to a slightly larger class of graphs. For a bipartite graph $G$ we denote the variable $r_{S}$ with $r_{S\cap V_1,S\cap V_2}$, where $(V_1,V_2)$ is the bipartition of $V(G)$.

\begin{theorem}\label{thm:bipartiteGrobner}
Let $G$ be a bipartite graph with parts $V_1$ and $V_2$. Then there is a square-free quadratic Gr\"obner basis of $I_{G\rightarrow \downspoon}$
given by the binomials
\[ r_{A,B}r_{C,D}-r_{A\cap C,B\cup D}r_{A\cup C,B\cap D} \]
where $A,C \subseteq V_1, B,D \subseteq V_2$ and both $A\cup B$ and $C\cup D$ are independent.
\end{theorem}
\begin{proof}
We will prove that this set of binomials generate the ideal by introducing a weight vector on the monomials, and finding a normal form.
That shows that our generating set is a Gr\"obner basis. In Lemma~\ref{lemma:multigrading} a technique to determine when a binomial is in the kernel using a multigrading was introduced. That technique will be used in this proof.

First we prove that if $r_{A,B},r_{C,D}\in R_{G \rightarrow \downspoon}$ then $r_{A\cup C,B\cap D}\in R_{G \rightarrow \downspoon}$. We need to check that $(A\cup C)\cup(B\cap D)$ is an independent set. The set $B\cap D$ cannot have any edges to $A$ since $A\cup B$ is independent. Similarly $B\cap D$ can not have any edges to $C$ and we can conclude that $(A\cup C)\cup(B\cap D)$ is independent. Using the same argument for $r_{A\cap C,B\cup D}$ we conclude that $r_{A\cap C,B\cup D}r_{A\cup C,B\cap D}\in R_{G \rightarrow \downspoon}$.

By computing the degrees of the monomials (including the degrees $d_v$) in $r_{A,B}r_{C,D}-r_{A\cap C,B\cup D}r_{A\cup C,B\cap D}$ we will use Lemma~\ref{lemma:multigrading} to establish that 
\[r_{A,B}r_{C,D}-r_{A\cap C,B\cup D}r_{A\cup C,B\cap D}\in I_{G\rightarrow \downspoon}.\] 
The degree $d_v(r_{A,B}r_{C,D})=2$ if $v$ is in both $A\cup B$ and $C\cup D$, and then $v$ is in both $(A\cap C)\cup(B\cup D)$ and $(A\cup C)\cup(B\cap D)$. The degree $d_v(r_{A,B}r_{C,D})=1$ if $v$ is in exactly one of $A\cup B$ and $C\cup D$, and then it is in exactly one of $(A\cap C)\cup(B\cup D)$ and $(A\cup C)\cup(B\cap D)$. Finally $d_v(r_{A,B}r_{C,D})=0$ if $v$ is in neither of $A\cup B$ and $C\cup D,$ and then $v$ is in neither of $(A\cap C)\cup(B\cup D)$ and $(A\cup C)\cup(B\cap D)$.

Any given monomial $m$ in $R_{G \rightarrow \downspoon}$ can be turned into normal form by Markov steps. That is, we want to find quadratic binomials $q_i$ that are of the type in the theorem statement, and monomials $n_i$ such that $m+q_1n_1+\cdots+ q_kn_k$ is a monomial $\prod_{i=1}^dr_{A'_i,B'_i}$ where $A'_i\subseteq A'_{i+1}$ and $B'_{i}\supseteq B'_{i+1}$. The normal form monomial is illustrated in Figure \ref{fig}.

Instead of a monomial in $R_{G \rightarrow \downspoon}$ we consider the ordered tuple of independent sets
\[(A_1\cup B_1,\ldots,A_k\cup B_k).\]
To move from
\[W=(A_1\cup B_1,\ldots,A_k\cup B_k)\]
to
\[
\begin{array}{rcl}
W'&=&(A_1\cup B_1,\ldots,(A_{t,1}\cap A_{t+1,1})\cup(B_{t,1}\cup B_{t+1,1}),\\
&&(A_{t}\cup A_{t+1})\cup(B_{t}\cap B_{t+1}),\ldots,A_k\cup B_k)
\end{array}
\]
corresponds to taking a Markov step of the type in the theorem statement:
\[
(r_{A_{t}\cap A_{t+1},B_{t}\cup B_{t+1}}r_{A_t\cup A_{t+1},B_t\cap B_{t+1}}-r_{A_{t},B_{t}}r_{A_{t+1},B_{t+1}})\prod_{i\in[k],i\notin\{t,t+1\}}r_{A_{i},B_{i}}.
\]
We denote this Markov step by $W\rightarrow_t W'$.
To any tuple 
$W=(A_1\cup B_1,\ldots,A_k\cup B_k)$ we associate the weight
$\omega(W)=\sum_{1\le i<j\le k}[(|A_j|-|A_i|)+(|B_j|-|B_i|)],$
or equivalently
\[\begin{array}{rcl}
\omega(W)&=&\displaystyle\sum_{j=1}^k(j-1)|A_j|-\sum_{i=1}^k(k-i-1)|A_i|\\
&&\displaystyle+\sum_{i=1}^k(k-i-1)|B_i|-\sum_{j=1}^k(j-1)|B_j|\\
&=&\displaystyle\sum_{j=1}^k((j-1)-(k-j-1))|A_j|+\sum_{i=1}^k((k-i-1)-(i-1))|B_i|\\
&=&\displaystyle\sum_{i=1}^k(2i-k)|A_i|+\sum_{i=1}^k(k-2i)|B_i|.
\end{array}
\]
If $W\rightarrow_tW'$ then 
\[
\begin{array}{rcl}
\omega(W)-\omega(W')&=&t(|A_t|-|A_t\cap A_{t+1}|)-t(|B_t|-|B_t\cup B_{t+1}|)\\&&+(t+1)(|A_{t+1}|-|A_t\cup A_{t+1}|)\\&&-(t+1)(|B_{t+1}|-|B_t\cap B_{t+1}|)\\&=&(|A_{t+1}|-|A_t\cup A_{t+1}|)+(|B_{t+1}|-|B_t\cap B_{t+1}|)\\&\le& 0
\end{array}
\]
with equality if and only if $A_{t+1}\supseteq A_{t}$ and $B_{t+1}\subseteq B_{t}$. If there are no $t$ such that $W\rightarrow_t W^*$ and $W\neq W^*$, then we can conclude that $W$ is on the normal form corresponding to Figure \ref{fig}. If there is a $t$ such that $W\rightarrow_t W^*$ and $W\neq W^*$ then $\omega(W')>\omega(W)$, but $\omega$ is a bounded integer, so we can only take a finite number of steps until we can find no more $t$, and then we have reached the normal form.

The normal form only depends on the numbers $d_v(m)$ and the degree of $m$, so if $m-n\in I_{G \rightarrow \downspoon}$ then both $m$ and $n$ have the same normal form and we can move between them using the Markov steps in the theorem statement.

\begin{figure}
\center
\includegraphics[width=\textwidth]{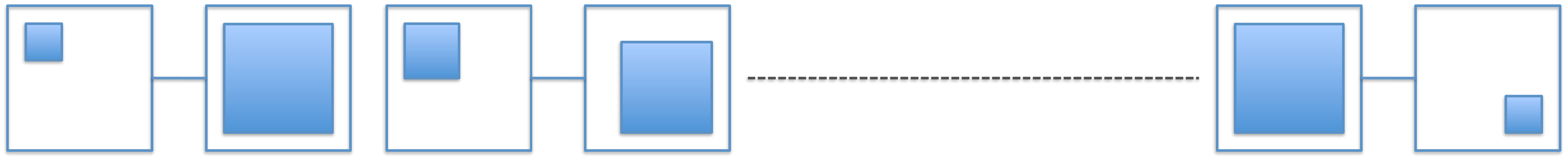}
\caption{The normal form of a monomial in $R_G$ as used in the proof of Theorem \ref{thm:bipartiteGrobner}. It is a tuple $(A_1\cup B_1,A_2\cup B_2,\ldots,A_k\cup A_k)$ corresponding to a monomial $r_{A_1,B_2}r_{A_2,B_2}\cdots r_{A_k,B_k}$ with $A_1\subseteq A_2\subseteq\ldots\subseteq A_k$ and $B_1\supseteq B_2\supseteq\ldots\supseteq B_k$.}\label{fig}
\end{figure}
\end{proof}

\begin{corollary}
If $G$ is a bipartite graph then $I_{G \rightarrow \downspoon}$ has a normal semigroup and is Cohen-Macaulay.
\end{corollary}
\begin{proof}
This follows from the theorem and Proposition~\ref{prop:squarefreeNormalCM}.
\end{proof}

\section{Independent sets from almost biparite graphs}
The previous section can be extended to a slightly larger class of graphs: the class of graphs that are bipartite if you delete a vertex. We call such graphs \emph{almost bipartite}. This set of graphs is interesting since it includes all cycles. The even cycles are bipartite but not the odd ones. Cycles are good models since they have a fairly simple structure and one can hope to understand what happens.

\begin{theorem}\label{t4}
If $G$ is almost bipartite then the ideal $I_{G\rightarrow \downspoon}$ has a quadratic square-free Gr\"obner basis.
\end{theorem}
The proof is quite technical and requires some new lemmas and some new notation.

\begin{figure}
\center
\includegraphics[width=\textwidth]{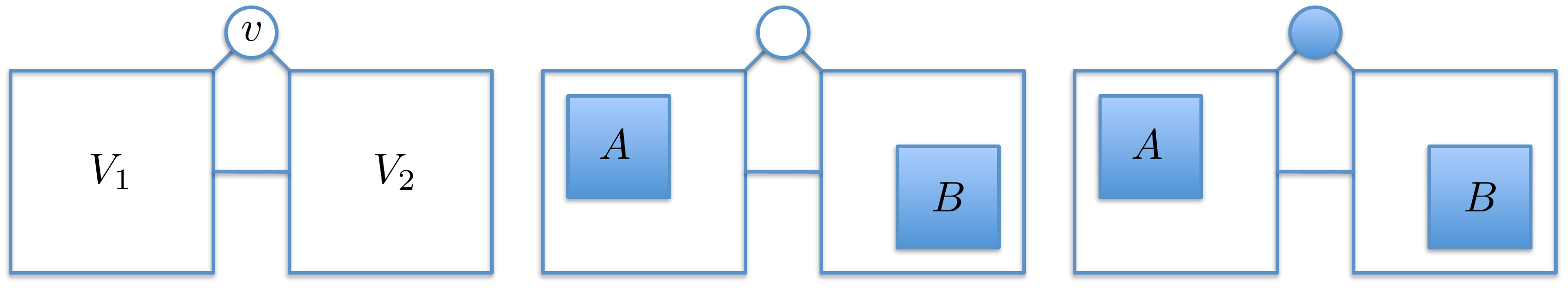}
\caption{Drawings of an almost bipartite graph with vertex set $V_1 \cup V_2 \cup \{v\},$ the variable $r^{\circ}_{A,B},$ and the variable $r^{\bullet}_{A,B}$.}\label{var}
\end{figure}

Let $G$ be a graph with $V(G)=V_1\cup V_2\cup \{v\}$ where the union is disjoint and $V_1$ and $V_2$ are independent sets. The variable $r_A$ is denoted $r^{\bullet}_{A\cap V_1,A\cap V_2}$ if $v\in A$, and $r^{\circ}_{A\cap V_1,A\cap V_2}$ if $v\notin A$.

A monomial $m$ is always of the form $m=m^\circ m^\bullet$, where $m^\circ$ only contains variables $r^\circ_{A,B}$ and $m^\bullet$ only contains variables $r^\bullet_{A,B}$. Define the degrees $\deg^\circ(m)=\deg(m^\circ)$ and $\deg^\bullet(m)=\deg(m^\bullet)$.

A binomial $r^\circ_{A,B}r^\circ_{C,D}-r^\circ_{A',B'}r^\circ_{C',D'}$ is \emph{uncovered} if it is in $I_{G\rightarrow \downspoon}$.

A binomial $r^\bullet_{A,B}r^\bullet_{C,D}-r^\bullet_{A',B'}r^\bullet_{C',D'}$ is \emph{covered} if it is in $I_{G\rightarrow \downspoon}$.

A binomial $r^\circ_{A,B}r^\bullet_{C,D}-r^\circ_{A',B'}r^\bullet_{C',D'}$ is \emph{mixed} if it is in $I_{G\rightarrow \downspoon}$.

A monomial  
\[
\prod_{i\in [d^\circ]} r^{\circ}_{A_i,B_i}\prod_{i\in [d^\bullet]} r^{\bullet}_{C_i,D_i}
\] 
in $R_{G\rightarrow \downspoon}$ is on \emph{intermediate normal form} if $A_i\subseteq A_{i+1},B_{i+1}\subseteq B_{i},C_{i+1}\subseteq C_{i},D_i\subseteq D_{i+1}$. This monomial is not unique in the sense that there can be two different monomials $m,m'$ both on intermediate normal form such that $m-m'\in I_{G\rightarrow \downspoon}$.

Recursively define the \emph{normal form} as follows. A monomial 
\[
\prod_{i\in [d^\circ]} r^{\circ}_{A_i,B_i}\prod_{i\in [d^\bullet]} r^{\bullet}_{C_i,D_i}
\] 
on intermediate normal form with $d^\circ=0$ is on normal form. Let $m$ be monomial and let  $n=\prod_{i\in [d^\circ]} r^{\circ}_{A_i,B_i}\prod_{i\in [d^\bullet]} r^{\bullet}_{C_i,D_i}$ be a monomial on intermediate normal form with $d^\circ>0$ such that $m-n\in I_{G\rightarrow \downspoon}$. If $|A_1|$ is minimal and $|B_1|$ is maximal among all such intermediate normal monomials, then $n$ is on normal form if $n/r^{\circ}_{A_1,B_1}$ is on normal form. We will show that this normal form monomial is unique in the sense that for any monomial $m$ there is only one normal form monomial $n$ such that $m-n\in I_{G\rightarrow \downspoon}$.

We will show that for each monomial $m$ there is a monomial $n$ on intermediate normal form so that $m-n,m^\circ-n^\circ$ and $m^\bullet-n^\bullet$ all are in the ideal generated by the uncovered and covered binomials. This is done in lemma~\ref{11}.

For each monomial on intermediate normal form it will be shown that there exists a mixed binomial bringing it closer to normal form. This is done in lemma~\ref{14}. This is all the machinery needed to prove the theorem.

We will draw the variables $r^{\circ}_{A,B}$ and $r^{\bullet}_{A,B}$ as in Figure~\ref{var}, the Markov steps of Lemma~\ref{8} (uncovered) and Lemma~\ref{9} (covered) are illustrated in Figures~\ref{l8} and~\ref{l9}.

We begin by some lemmas describing the needed binomials of the different types.
\begin{lemma}\label{8}
If $r^{\circ}_{A,B},r^{\circ}_{C,D}\in R_{G\rightarrow \downspoon}$ then $r^{\circ}_{A\cap C,B\cup D}r^{\circ}_{A\cup C,B\cap D}-r^{\circ}_{A,B}r^{\circ}_{C,D}\in I_{G\rightarrow \downspoon}$. In other words: $r^{\circ}_{A\cap C,B\cup D}r^{\circ}_{A\cup C,B\cap D}-r^{\circ}_{A,B}r^{\circ}_{C,D}$ is an uncovered binomial.
\end{lemma}
\begin{proof}
Since $v$ is not in any of the independent sets defining the variables $r^{\circ}_{A,B}$ or $r^{\circ}_{C,D}$, we can ignore it and proceed as in the bipartite case Theorem~\ref{thm:bipartiteGrobner}.
\end{proof}

\begin{figure}
\center
\includegraphics[width=\textwidth]{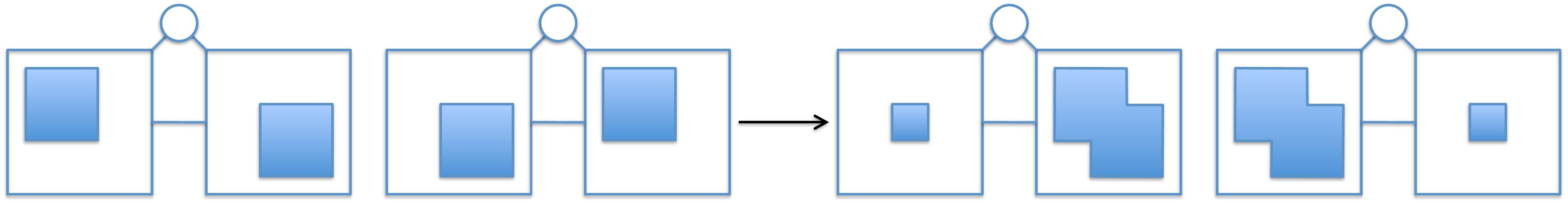}
\caption{A Markov step in Lemma~\ref{8}.} \label{l8}
\end{figure}
\begin{lemma}\label{9}
If $r^{\bullet}_{A,B},r^{\bullet}_{C,D}\in R_{G\rightarrow \downspoon}$ then $r^{\bullet}_{A\cup C,B\cap D}r^{\bullet}_{A\cap C,B\cup D}-r^{\bullet}_{A,B}r^{\bullet}_{C,D}\in I_{G\rightarrow \downspoon}$. In other words $r^{\bullet}_{A\cup C,B\cap D}r^{\bullet}_{A\cap C,B\cup D}-r^{\bullet}_{A,B}r^{\bullet}_{C,D}$ is a covered binomial.
\end{lemma}
\begin{proof}
Since no elements in $N(v)$ is in any of the sets defining the variables $r^{\bullet}_{A,B},r^{\bullet}_{C,D}$ we can ignore them together with $v$  and proceed as in the bipartite case Theorem~\ref{thm:bipartiteGrobner}.
\end{proof}

\begin{figure}
\center
\includegraphics[width=\textwidth]{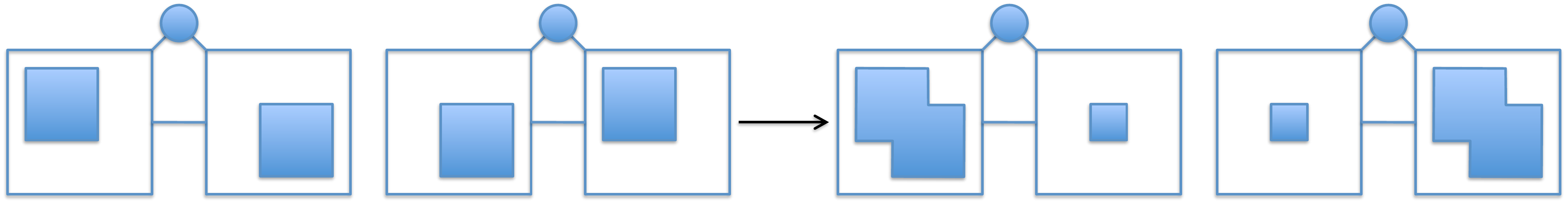}
\caption{A Markov step in Lemma \ref{9}.} \label{l9}
\end{figure}

The two previous lemmas give two similar types of generators for $I_{G\rightarrow \downspoon}$, but we will also need another of a  quite different type.

\begin{lemma}\label{10}
If $r^{\circ}_{A,B},r^{\bullet}_{C,D},r^{\circ}_{A\setminus E,B\cup (N(E)\cap D)},r^{\bullet}_{C\cup E,D\setminus N(E)}\in R_{G\rightarrow \downspoon}$, $E\subseteq A$, and $E\cap C=\emptyset$, then
\[r^{\circ}_{A,B}r^{\bullet}_{C,D}-r^{\circ}_{A\setminus E,B\cup (N(E)\cap D)}r^{\bullet}_{C\cup E,D\setminus N(E)}\] is in $I_{G\rightarrow \downspoon}$. In other words $r^{\circ}_{A,B}r^{\bullet}_{C,D}-r^{\circ}_{A\setminus E,B\cup (N(E)\cap D)}r^{\bullet}_{C\cup E,D\setminus N(E)}$ is a mixed binomial. This Markov step is drawn in Figure~\ref{special}.
\end{lemma}

\begin{proof}
Since we assumed that all the indeterminates are in $R_{G\rightarrow \downspoon}$, we just check the multidegrees by Lemma~\ref{lemma:multigrading}, of each vertex $d_u(m)$ for the monomials $r^{\circ}_{A\setminus E,B\cup (N(E)\cap D)}r^{\bullet}_{C\cup E,D\setminus N(E)}$ and $r^{\circ}_{A,B}r^{\bullet}_{C,D}.$  The indeterminates in the ring $R_{G\rightarrow \downspoon}$ correspond to independent sets in the graph $G$, and it is assumed that all indeterminates in the statement of the lemma are in the ring $R_{G\rightarrow \downspoon}$. For certain sets $E$ the set $C\cup E\cup (D\setminus N(E) )\cup\{v\}$ might not be independent, but that situation is not covered by this lemma. We do not have to worry about independence of the sets corresponding to the indeterminates since they are assumed to be in $R_{G\rightarrow \downspoon}$.

The degree $d_v(m)$ is $1$ for both monomials. The number $d_{u}(m)$ is $1$ for both monomials when $u\in A\Delta C$, and similarly $d_{u}(m)=1$ when $u\in B\Delta D$.  When $u\in A\cap C$ then $d_{u}(m)=2$ for both monomials. The set $A\cup B$ is independent, this implies that $E\cup B$ is independent. If $u\in B\cap D$ then $u\notin N(E)$ since $E\cup (B\cap D)$ is independent. The conclusion is that $d_u(M)=2$ for both monomials if $u\in B\cap D$. Finally $d_u(m)=0$ for both monomials if $u$ is in none of the sets $A,B,C,D$.
\end{proof}
Note that in Lemma~\ref{10} not all subsets of $A$ can be used as a set $E$, it is required that $E\cap N(v)=\emptyset$.

\begin{figure}
\center
\includegraphics[width=\textwidth]{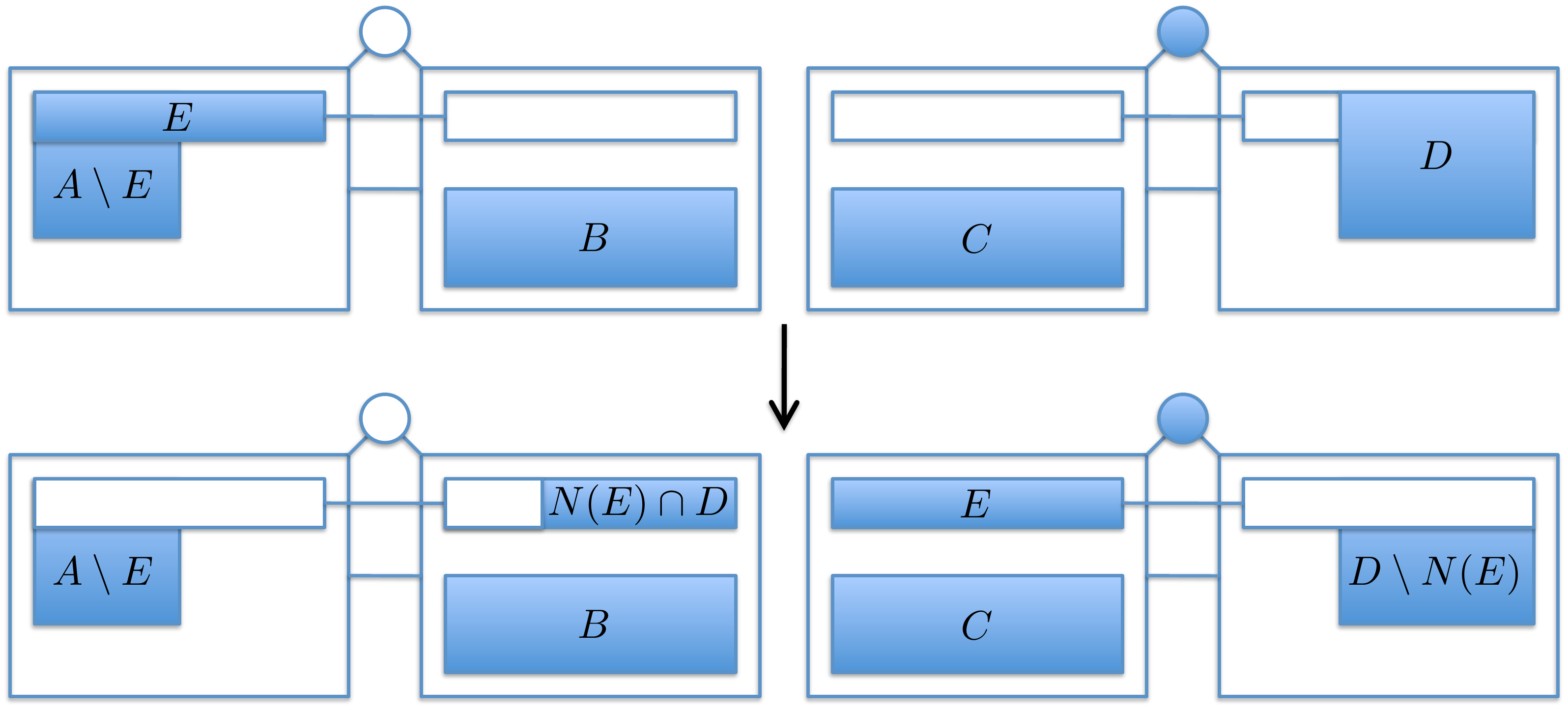}
\caption{The Markov step of Lemma \ref{10}, $r^{\circ}_{A,B}r^{\bullet}_{C,D}\rightarrow r^{\circ}_{A\setminus E,B\cup(N(E)\cap D)}r^{\bullet}_{C\cup E,D\setminus N(E)}$.}\label{special}
\end{figure}

As in the case with bipartite graphs there is a normal form that we want to reach. A first step towards this is  the following lemma.

\begin{lemma}\label{11}
Let $m$ be a monomial in $R_{G\rightarrow \downspoon}$. Then there is a monomial $n$ on intermediate normal form so that $m-n$ is in the ideal generated by the covered and uncovered binomials.
\end{lemma}
\begin{proof}
First it will be proved that there is a monomial $n^\circ$ on intermediate normal form such that $m^\circ-n^\circ$ is in the ideal generated by uncovered binomials. And similarly that there is a monomial $n^\bullet$ on intermediate normal form such that $m^\bullet-n^\bullet$ is in the ideal generated by covered binomials.

In fact, when reasoning about $m^\circ$ we can ignore $v$ and proceed as in Theorem~\ref{thm:bipartiteGrobner}. And similarly when reasoning about $m^\bullet$, we can ignore $v$ together with $N(v)$. The normal forms reached in the proof of Theorem~\ref{thm:bipartiteGrobner} are then exactly the intermediate normal forms wanted. Recall that in the bipartite case the ideal was generated by binomials $r_{A,B}r_{C,D}-r_{A\cap C,B\cup D}r_{A\cup C,B\cap D}$, and the normal form satisfied the same type of inclusions.

Now
\[m-n=m^\circ m^\bullet-n^\circ n^\bullet=(m^\circ-n^\circ)m^\bullet+(m^\bullet-n^\bullet)n^\circ,\]
so $m-n$ is in the ideal generated by covered and uncovered binomials, since $(m^\circ-n^\circ)$ and $(m^\bullet-n^\bullet)$ are.
\end{proof}

\begin{figure}
\center
\includegraphics[width=\textwidth]{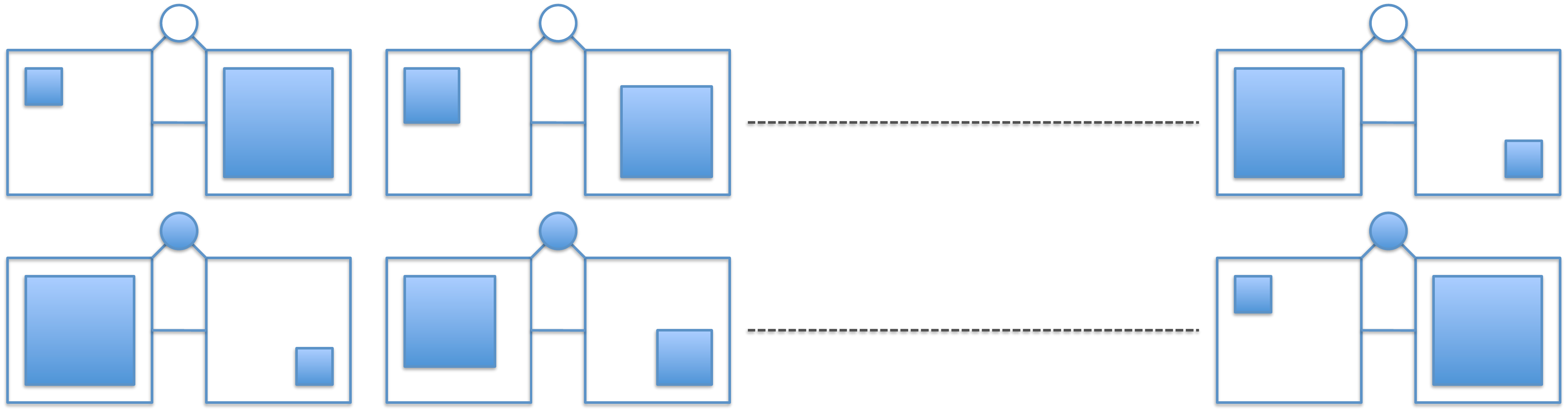}
\caption{A monomial on intermediate normal form in Lemma \ref{11}.} \label{fff}
\end{figure}

Now we will begin to show the existence of some important sets that later will be used to show the existence of the needed mixed generators in Lemma~\ref{10}.

\begin{lemma}\label{13}
Let $m=\prod_{j=1}^{d^\circ} r^{\circ}_{A_j,B_j}\prod_{j=1}^{d^\bullet} r^{\bullet}_{C_j,D_j}$ and $n=\prod_{j=1}^{d^\circ} r^{\circ}_{A'_j,B'_j}\prod_{j=1}^{d^\bullet} r^{\bullet}_{C'_j,D'_j}$ be monomials on intermediate normal form, and $n-m\in I_{G\rightarrow \downspoon}$. If $F=A_1\setminus A'_1\neq\emptyset$, then there is an $i\in[d^\bullet]$ so that $F\nsubseteq C_i$ and $(N(F)\setminus B'_1)\cap D_i=\emptyset$.
\end{lemma}
\begin{proof}
Recall from the definition of intermediate normal form that
\[A_{j}\subseteq A_{j+1},A'_{j}\subseteq A'_{j+1},\,\,\,1\le j<d^\circ\]
\[B_{j}\supseteq B_{j+1},B'_{j}\supseteq B'_{j+1},\,\,\,1\le j<d^\circ\]
\[C_{j}\supseteq C_{j+1},C'_{j}\supseteq C'_{j+1},\,\,\,1\le j<d^\bullet\]
\[D_{j}\subseteq D_{j+1},D'_{j}\subseteq D'_{j+1},\,\,\,1\le j<d^\bullet.\]

The first step is to show that there is an $i$ so that $(N(F)\setminus B'_1)\cap D'_i=\emptyset$ and $F$ is not a subset of $C_i$. This will be shown by contradiction.

Assume that $F$ is a subset of $C_i$ whenever $(N(F)\setminus B'_1)\cap D'_i=\emptyset$. The set $F$ is a subset of all sets $A_j$ since it is a subset of $A_1$ and $A_1$ is a subset of all $A_j$. The set $F$ is disjoint from $A'_1$ and have to be a subset of at least one more set $C'_j$ than $C_j$.  In particular $F$ has to be a subset of a $C'_i$ with $(N(F)\setminus B'_1)\cap D'_i\neq\emptyset$. 
Remember that the set $C'_i\cup D'_i$ is independent. This is a contradiction since $C'_i\cup D'_i$ contains both $F$ and parts of $N(F)$. Hence $F$ is not a subset of all $C_i$ such that $(N(F)\setminus B'_1)\cap D'_i=\emptyset$.

We are done if we prove that $(N(F)\setminus B'_1)\cap D_i$ is empty if and only  if $(N(F)\setminus B'_1)\cap D'_i$ is empty, and that is our last step.

The neighborhood of $F$ have no elements in common with $B_1$ since $A_1\cup B_1$ is an independent set containing $F$. It follows that the set $N(F)\setminus B'_1$ is disjoint from both $B_1$ and $B'_1$.

Let $u$ be an element in $N(F)\setminus B'_1$. If $d_u(m)>0$ then $u$ is in some sets $B_l$ and $D_l$ but no sets $A_l$ and $C_l$, and $u$ will be in some sets $B'_l$ and $D'_l$ but no sets $A'_l$ and $C'_l$. Remember that $B_{l+1}\subseteq B_l$ and $B'_{l+1}\subseteq B'_l$, this implies that $u$ can not be in any set $B_l$ or $B'_l$. Recall that $D_l\subseteq D_{l+1}$ and  $D'_l\subseteq D'_{l+1}$. The conclusion is that If $d_u(m)>0$ then $u$ has to be in the last $d_u(m)$ sets $D_l$ and $D'_l$. The element $u$ was arbitrary so $(N(F)\setminus B'_1)\cap D_j=(N(F)\setminus B'_1)\cap D'_j$.
In particular $(N(F)\setminus B'_1)\cap D_j=\emptyset$ if and only if $(N(F)\setminus B'_1)\cap D'_j=\emptyset$.
\end{proof}

One important property of the sets in Lemma~\ref{13} is that $F$ will never contain vertices adjacent to $v$. This will be proved as part of the next lemma which is the main tool in the proof of Theorem~\ref{t4}.

\begin{lemma}\label{14}
Let 
\[
m=\prod_{i\in [d^\circ]} r^{\circ}_{A_i,B_i}\prod_{i\in [d^\bullet]} r^{\bullet}_{C_i,D_i}
\] 
and 
\[
n=\prod_{i\in [d^\circ]} r^{\circ}_{A'_i,B'_i}\prod_{i\in [d^\bullet]} r^{\bullet}_{C'_i,D'_i}
\] 
be monomials on intermediate normal form and $m-n\in I_{G\rightarrow \downspoon}$.
If $F=A_1\setminus A'_1\neq \emptyset$ then there is a non-empty subset $E$ of $F$ such that 
\[
r^{\circ}_{A_1\setminus E,B_1\cup (N(E)\cap D_i)}r^{\bullet}_{C_i\cup E,D_i\setminus N(E)}-r^{\circ}_{A_1,B_1}r^{\bullet}_{C_i,D_i}
\]
is a mixed binomial for some $i\in[d^\bullet]$.
\end{lemma}
\begin{proof}
Pick the $i$ from Lemma~\ref{13}. That is $i\in[d^\bullet]$ so that $F\nsubseteq C_i$ and $(N(F)\setminus B'_1)\cap D_i=\emptyset$. Now set $E=F\setminus C_i$.

It remains to be proven that $(A_1\setminus E)\cup B_1\cup (N(E)\cap D_i)$ and $C_i\cup E\cup(D_i\setminus N(E))\cup\{v\}$ are independent.

The sets $A_i$ and $A'_i$ satisfies $A_i\subseteq A_{i+1}$ and $A'_i\subseteq A'_{i+1}$. Note that this implies that any element $u$ with  $u\in A_1$ and $u\nin A'_1$ is in some set $C'_j$, otherwise the degree $d_u$ would be different for the two monomials $m$ and $n$. The elements in the sets $C'_j$ can not be adjacent to $v$ since the sets $\{v\}\cup C'_j\cup D'_j$ are independent and so $(A_1\setminus A'_1)\cup\{v\}$ is independent. Together with the fact that $C_i\cup D_i\cup \{v\}$ is independent this proves that $C_i\cup E\cup(D_i\setminus N(E))\cup\{v\}$ is independent.

We should verify that $(A_1\setminus E)\cup B_1\cup (N(E)\cap D_i)$ is independent, and indeed it is since $(N(E)\cap D_i)\subseteq B'_1$ and $B'_1\cup (A'_1\cap A_1)$ are independent.

The polynomial $r^{\circ}_{A_1\setminus E,B_1\cup (N(E)\cap D_i)}r^{\bullet}_{C_i\cup E,D_i\setminus N(E)}-r^{\circ}_{A_1,B_1}r^{\bullet}_{C_i,D_i}$ is of the type in Lemma~\ref{10} and all the corresponding sets are independent.
\end{proof}
This show that the normal form is unique, since if two different sets $A_1$ and $A'_1$ are minimal in the sense of the definition of normal form, then $A_1\setminus A'_1\neq\emptyset$. We can then use the binomials in Lemma~\ref{14} to reach smaller sets $A$, and the sets $A_1$ and $A'_1$ were not minimal.

Now we can finish the proof of the main theorem of this section.
\begin{proof}[Proof of Theorem~\ref{t4}]
We will prove that by using Markov steps of degree $2$ it is possible to go from any monomial to a monomial on normal form. Let $m$ be any monomial, the proof will be by induction on $\deg^\circ(m)$. 

Starting from any monomial we can reach a monomial on intermediate normal form using only the covered and uncovered Markov steps, according to Lemma~\ref{11}. The base case of the induction $\deg^\circ(m)=0$ then follows from the fact that in this case the normal form is the intermediate normal form.

Again starting from any monomial we can reach a monomial on intermediate normal form using only the covered and uncovered Markov steps. Let $m=\prod_{i\in [d^\circ]} r^{\circ}_{A_i,B_i}\prod_{i\in [d^\bullet]} r^{\bullet}_{C_i,D_i}$ be the reached intermediate normal monomial.

The induction step will be proved by demonstrating how to go from $m$ to a intermediate normal form monomial satisfying the normal form minimality required for $|A_1|$ and the maximality required for  $|B_1|$.

We will show that if $|A_1|$ do not satisfy the required minimality then there is a Markov step that takes $m$ to a monomial $n$. The only difference between $m^\circ$ and $n^\circ$ is that $n^\circ$ contains $r^\circ_{A^*_1,B^*_1}$ instead of $r^\circ_{A_1,B_1}$, where $A^*_1\subset A_1$ and $B_1\subseteq B^*_1$. This then makes it possible to assume that $A_1$ satisfies the desired minimality, after this a similar argument is used to prove that we can get a maximal $B_1$.

Assume that $A_1$ do not satisfy the minimality in the definition of normal monomial. Then there is another intermediate normal monomial $m'=\prod_{i\in [d^\circ]} r^{\circ}_{A'_i,B'_i}\prod_{i\in [d^\bullet]} r^{\bullet}_{C'_i,D'_i}$ such that  $m-m'\in I_{G\rightarrow \downspoon}$ and $A_1\setminus A'_1\neq \emptyset$. This is the situation covered in Lemma~\ref{14}. The Markov step 
\[
r^{\circ}_{A_1\setminus E,B_1\cup (N(E)\cap D_i)}r^{\bullet}_{C_i\cup E,D_i\setminus N(E)}-r^{\circ}_{A_1,B_1}r^{\bullet}_{C_i,D_i}
\] 
from Lemma~\ref{14}  can be used to go from $m$ to $n$ with 
\[
r^\circ_{A^*_1,B^*_1}=r^{\circ}_{A_1\setminus E,B_1\cup (N(E)\cap D_i)}.
\] 
Now we can assume that $A_1$ satisfies the minimality required for normal form.

The argument to get $B_1$ maximal is similar. The big difference is\linebreak that we might also need mixed Markov steps going from $r^\circ_{A,B}r^\bullet_{C,D\cup \{u\}}$ to $r^\circ_{A,B\cup \{u\}}r^\bullet_{C,D}$.

Recall that  $m=\prod_{i\in [d^\circ]} r^{\circ}_{A_i,B_i}\prod_{i\in [d^\bullet]} r^{\bullet}_{C_i,D_i}$ is on intermediate normal form and $A_1$ is normal form minimal. If $B_1$ is not normal form maximal then some $B_i$ or $D_i$ contains elements that can be added to $B_1$ without breaking the independence of $A_1\cup B_1$. Adding elements like this can then be done by using Markov steps from $r^\circ_{A_1,B_1}r^\circ_{A_i,B_i}$ to $r^\circ_{A_1\cap A_i,B_1\cup B_i}r^\circ_{A_i\cup A_i,B_1\cap B_i}$ or Markov steps from $r^\circ_{A,B}r^\bullet_{C,D\cup \{u\}}$ to $r^\circ_{A,B\cup \{u\}}r^\bullet_{C,D}$. After using such Markov steps the monomial might not be in intermediate normal form. By Lemma~\ref{11} it is still possible to reach an intermediate normal form, this time with the larger $B_1$.

Now it is possible to reach an intermediate normal form satisfying the minimality and maximality required for normal form. By induction it is possible to go from $m/r^{\circ}_{A_i,B_i}$ to the corresponding normal form using the same degree $2$ Markov steps. Together this gives that it is possible to always reach the normal form using the quadratic square-free Markov steps.
\end{proof}

\begin{corollary}
If $G$ is an almost bipartite graph, then $I_{G \rightarrow \downspoon}$ has a normal semigroup and is Cohen-Macaulay.
\end{corollary}
\begin{proof}
This follows from Theorem~\ref{t4} and Proposition~\ref{prop:squarefreeNormalCM}.
\end{proof}

\begin{example}
When cycles are not complete graphs they are generated in degree $2$ according to Theorem~\ref{t4}. The first example is $C_4$ with edges $\{12,23,34,14\}$. This cycle is bipartite and we get the generators $r_{\{1,3\}}r_{\emptyset}-r_{\{1\}}r_{\{3\}}$ and $r_{\{2,4\}}r_{\emptyset}-r_{\{2\}}r_{\{4\}}.$

The next example is $C_5$ with edges $\{12,23,34,45,15\}$. This cycle is not bipartite, but if we delete the vertex $1$ it is. The uncovered generators are
$r^\circ_{\{2,4\},\emptyset}r^\circ_{\emptyset,\emptyset}-r^\circ_{\{2\},\emptyset}r_{\{4\},\emptyset}$ and $r_{\emptyset ,\{3,5\}}r_{\emptyset,\emptyset}-r_{\emptyset,\{3\}}r_{\emptyset,\{5\}}.$
 In this case there are no covered binomials needed to generate the ideal. The mixed generators needed are $r^\bullet_{\emptyset,\emptyset}r^\circ_{\{2,4\},\emptyset}-r^\bullet_{\{4\},\emptyset}r^\circ_{\{2\},\emptyset},r^\bullet_{\emptyset,\emptyset}r^\circ_{\emptyset,\{3,5\}}-r^\bullet_{\emptyset,\{3\}}r^\circ_{\emptyset,\{5\}}$ and $r^\bullet_{\{\emptyset,3\}}r^\circ_{\{4\},\emptyset}-r^\bullet_{\{4\},\emptyset}r^\circ_{\emptyset,\{3\}}.$
\end{example}

\section{Polytopes}\label{sec:polytopes}

A \emph{polytope} is the convex hull of a finite set of points $S=\{s_1,\ldots,s_m\}$ in $\mathbb{R}^n$, or equivalently a bounded set consisting of the points satisfying finitely many linear inequalities.  A polytope $P_A$ is associated with the toric ideal $I_A$ where $A$ is a matrix: the polytope is the convex hull of the columns of $A$. The polytope associated with $I_{G\rightarrow H}$ is denoted $P_{G\rightarrow H}$. We give an explicit description of $P_{G\rightarrow H}$ in an independent definition.

\begin{definition}
If $G$ and $H$ are graphs, and 
\[ E = \{ (e,\rho) \mid e\in E(G) \mathrm{\ and\ } \rho : G[e] \rightarrow H \mathrm{\ is\ a\ graph\ homomorphism} \}, \]
then the \emph{polytope of graph homomorphisms from $G$ to $H$}, $P_{G\rightarrow H}$, is the convex hull in $\mathbf{R}^E$
of points $x_\phi$ indexed by graph homomorphisms $\phi$ from $G$ to $H$ as
\[ x_\phi \cdot \mathbf{e}_{e,\rho} = \left\{ 
\begin{array}{cl}
1 & \phi|_{G[e]} = \rho, \\
0 & \phi|_{G[e]} \neq \rho. \\
\end{array}
   \right. \]
\end{definition}

A good reference for polytope theory is \cite{ziegler1995}. This lemma is implicit in \cite{sturmfelsSullivant2008}.

\begin{lemma}\label{15}
If $A$ and $B$ are matrices that give homogeneous toric ideals $I_A$ and $I_B$, and $P_B$ is a face of $P_A$, then $\mu(I_B)\le\mu(I_A)$.
\end{lemma}
\begin{proof}
First we observe that $I_B$ is a natural subset of $I_A$, since we can obtain $B$ by removing columns of $A$. If $x^u-x^v$ is a generator of $I_B$, then $\sum_{i=1}^d x^{w_i}(x^{u_i}-x^{v_i})=x^u-x^v$ where $x^{u_i}-x^{v_i}$ for $i\in[d]$ are generators of $I_A$ and $x^{w_i}$ is some monomial. What could go wrong is that $x^{u_i}-x^{v_i}$ is not in $I_B$ for some $i$. We know that $w_i,u_i,u\ge 0$. We have that $Bu=A(u_i+w_i)$. We also have that $A(u_i+w_i)/\mathrm{deg}(x^{u_i+w_i})$ is a point in $P_A$. The point $A(u_i+w_i)/\mathrm{deg}(x^{u_i+w_i})$ must be in $P_B$ since it is $A(u_i+w_i)/\mathrm{deg}(x^{u})=Bu/\mathrm{deg}(x^{u})$. Hence $u+w_i$ can only be nonzero on entries corresponding to the vertices of the facet $B$.
\end{proof}

This is useful since the geometry of polytopes sometimes is easier to understand than the algebra. It turns out that when we look at ideals from graph homomorphisms $I_{G\rightarrow H}$, certain minors of $H$ have an interpretation in the polytope.
\begin{lemma}\label{16}
If $G$ and $H$ are graphs, and $v$ is a vertex of $H$, then $P_{G\rightarrow H\setminus v}$ is a face of $P_{G\rightarrow H}$.
\end{lemma}
\begin{proof}
Intersect $P_{G\rightarrow H}$ with the hyperplanes $x_{e,\phi}=0$ if $v\in\phi(e)$. The resulting polytope is the convex hull of the vectors coming from homomorphism where no vertex is mapped to $v$, that is $P_{G\rightarrow H \setminus v}$. The intersection is a face since $P_{G \rightarrow H}$ is a $0/1$-polytope.
\end{proof}

\begin{lemma}\label{t8}
If $G$ and $H$ are graphs, and $e$ is an edge of $H$, then $P_{G\rightarrow H\setminus e}$ is a face of $P_{G\rightarrow H}$.
\end{lemma}
\begin{proof}
This is similar to the proof of Lemma~\ref{16} but with the hyperplanes $x_{f,\phi}=0$ when $\phi(f)=e$. 
\end{proof}

\begin{theorem}\label{theo:faceInclusion}
If $G$, $H_1$ and $H_2$ are graphs with $H_1 \subseteq H_2$, then $P_{G\rightarrow H_1}$ is a face of $P_{G\rightarrow H_2}$
\end{theorem}
\begin{proof}
The graph $H_1$ can be reduced to $H_2$ be removing vertices and edges. By repeated use of Lemma~\ref{16} and Lemma~\ref{t8} we are done.
\end{proof}

Theorem~\ref{theo:faceInclusion} and Lemma~\ref{15} combined gives an alternative and nice proof of Corollary~\ref{corollary:restrictMarkovBases}, that $\mu(P_{G\rightarrow H_1}) \leq \mu(P_{G\rightarrow H_2})$. Contracting an edge is in general not possible in any nice way, for example the polytope $P_{C_3\rightarrow C_4}$ is empty but $P_{C_3\rightarrow C_3}$ is not, and neither is the polytope $P_{C_3\rightarrow C'_3}$ where $C'_3$ is $C_3$ with one extra loop. 

\subsection{Stable set polytopes}
In optimization theory it is more common to refer to independent sets as stable sets. The \emph{stable set polytope} of a graph $G$ is a polytope in $\mathbb{R}^{V(G)}$ defined as the convex hull of the points $x_S$, indexed by stable sets $S$ of $G$ as
\[ x_S \cdot \mathbf{e}_v = \left\{ 
\begin{array}{cl}
1 & v \in S, \\
0 & v \not\in S. \\
\end{array}
   \right. \]

The stable set polytope always has the following inequalities among its defining inequalities: $0\le x_i\le1$ and $\sum_{i\in K} x_i\le 1$ when $K$ is a maximal complete subgraph of $G$. It is known that these are the only defining inequalities when the graph $G$ is perfect \cite{Grotschel}. Another type of defining inequality that graphs containing odd holes (that is, induced odd cycles) has is $\sum_{i\in C}x_i\le \lfloor(|C|-1)/2\rfloor$ where $C$ is an odd hole of $G$. These are the defining inequalities for a large class of graphs, including the $K_4$-minor free graphs \cite{mahjoub1988}. Many important optimization problem can be stated as minimizing a linear form over a stable set polytope. It is therefore of great interest to understand their facet structure.

\begin{proposition}\label{prop:stableSetPolytope}
The polytope of graph homomorphisms $P_{G \rightarrow \downspoon}$ is isomorphic to the stable set polytope of $G.$
\end{proposition}
\begin{proof}
This proof builds on the same basic idea as that of Lemma~\ref{lemma:multigrading}.
Let 
\[ E  = \{ (e,\rho) \mid e\in E(G) \mathrm{\ and\ } \rho : G[e] \rightarrow \downspoon \mathrm{\ is\ a\ graph\ homomorphism} \}. \]
The polytope $P_{G \rightarrow \downspoon}$ in $\mathbb{R}^E$ is the convex hull of points $x_\phi$ indexed by graph homomorphisms $\phi$ from $G$ to $\downspoon$ as
\[ x_\phi \cdot \mathbf{e}_{e,\rho} = \left\{ 
\begin{array}{cl}
1 & \phi|_{G[e]} = \rho, \\
0 & \phi|_{G[e]} \neq \rho. \\
\end{array}
\right. \]
Any graph homomorphism $\rho: G[e] \rightarrow \downspoon$ can be encoded by the edge $e$ (as a subset of $V(G)$) and the subset $s$
of $e$ that is mapped onto the unlooped vertex of $\downspoon$. That is,
\[ E = \{ (e,s)  \mid e\in E(G),\, s\subset e,\, |s|\leq 1 \}. \]
The graph homomorphisms $\phi$ from $G$ to $\downspoon$ can similarly be encoded by the independent sets $S$ of $G$. So, $P_{G \rightarrow \downspoon}$ is the convex hull of points $x_S$ with $S$ independent, defined by
\[ x_S \cdot \mathbf{e}_{e,s} = \left\{ 
\begin{array}{cl}
1 & S\cap e = s, \\
0 & S\cap e \neq s. \\
\end{array}
\right. \]
If $v$ is a vertex of $G$, then the value of $x_S \cdot \mathbf{e}_{e,\{v\}}$ is the same for all edges $e$ containing $v$, by the same argument as in
the proof of Lemma~\ref{lemma:multigrading}. The values of all $x_S \cdot \mathbf{e}_{e,\emptyset}$ are determined by that the ideals of graph homomorphisms are homogeneous. Thus the isomorphism from the polytope $P_{G \rightarrow \downspoon}$ to stable set polytope of $G$ is defined by sending $ \mathbf{e}_{e,\{v\}}$ to $\mathbf{e}_v$ for all vertices $v$ of $G$. 
\end{proof}

A $d$-dimensional polytope is \emph{simple} if every vertex is in exactly $d$ facets. If a polytope $P$ associated to a toric variety $X$ is not simple, then the variety is not smooth, as explained in Section 2.1 of \cite{fulton1993}.

\begin{example}\label{ex:polyC4Ind}
In this example we present the polytope $P_{C_4 \rightarrow \downspoon}$. This is the matrix defining the toric variety $X_{C_4 \rightarrow \downspoon},$ with the columns indexed by the independent sets of $C_4,$ and the rows indexed as in the proof of Proposition~\ref{prop:stableSetPolytope}.
\[
\begin{array}{c|ccccccc}
&{\emptyset}&{\{1\}}&{\{2\}}&{\{3\}}&{\{4\}}&{\{1,3\}}&{\{2,4\}}\\
\hline
{12,\emptyset}&1&0&0&1&1&0&0\\
{12,\{1\}}&0&1&0&0&0&0&1\\
{12,\{2\}}&0&0&1&0&0&1&0\\
{23,\emptyset}&1&1&0&0&1&0&0\\
{23,\{2\}}&0&0&1&0&0&1&0\\
{23,\{3\}}&0&0&0&1&0&0&1\\
{34,\emptyset}&1&1&1&0&0&0&0\\
{34,\{3\}}&0&0&0&1&0&0&1\\
{34,\{4\}}&0&0&0&0&1&1&0\\
{14,\emptyset}&1&0&1&1&0&0&0\\
{14,\{1\}}&0&1&0&0&0&0&1\\
{14,\{4\}}&0&0&0&0&1&1&0\\
\end{array}.
\]
The polytope $P_{C_4 \rightarrow \downspoon}$ in $\mathbb{R}^{12}$ is the convex hull of the seven column vectors of the matrix. The isomorphic
stable set polytope of $C_4$ is given by only remembering one row for each vertex of $G$:
\[
\begin{array}{c|ccccccc}
&{\emptyset}&{\{1\}}&{\{2\}}&{\{3\}}&{\{4\}}&{\{1,3\}}&{\{2,4\}}\\
\hline
\{1\}&0&1&0&0&0&0&1\\
\{2\}&0&0&1&0&0&1&0\\
\{3\}&0&0&0&1&0&0&1\\
\{4\}&0&0&0&0&1&1&0\\
\end{array}.
\]
Using the polymake software \cite{gawrilowJoswig2000} we get that the polytope has eight facets and they are spanned by these collections of independent sets of $C_4$:
\[ 
\begin{array}{cc}
\{\{1\},\{2\},\{1,3\},\{2,4\}\} & \{\{2\},\{3\},\{1,3\},\{2,4\}\}\\
 \{\{1\},\{4\},\{1,3\},\{2,4\}\} & \{\{3\},\{4\},\{1,3\},\{2,4\}\} \\
  \{\emptyset, \{1\},\{2\},\{3\},\{1,3\}\} &   \{\emptyset, \{1\},\{3\},\{4\},\{1,3\}\} \\
  \{\emptyset, \{1\},\{2\},\{4\},\{2,4\}\} &   \{\emptyset, \{2\},\{3\},\{4\},\{2,4\}\} \\
\end{array}.
\]
The vertices of the independent sets $\emptyset, \{1\}, \{2\}, \{3\}, \{4\}$ are in four facets, and those of $\{1,3\}$ and $\{2,4\}$ are in six facets.
The toric variety $X_{C_4 \rightarrow \downspoon}$ is not smooth since the polytope $P_{C_4 \rightarrow \downspoon}$ is not simple, which also can be seen from an easy Jacobian calculation. In Figure~\ref{fig:Schlegel} is a Schlegel diagram of $P_{C_4\rightarrow \downspoon}$ drawn.
\end{example}
\begin{figure}
\center
\includegraphics[scale=0.5]{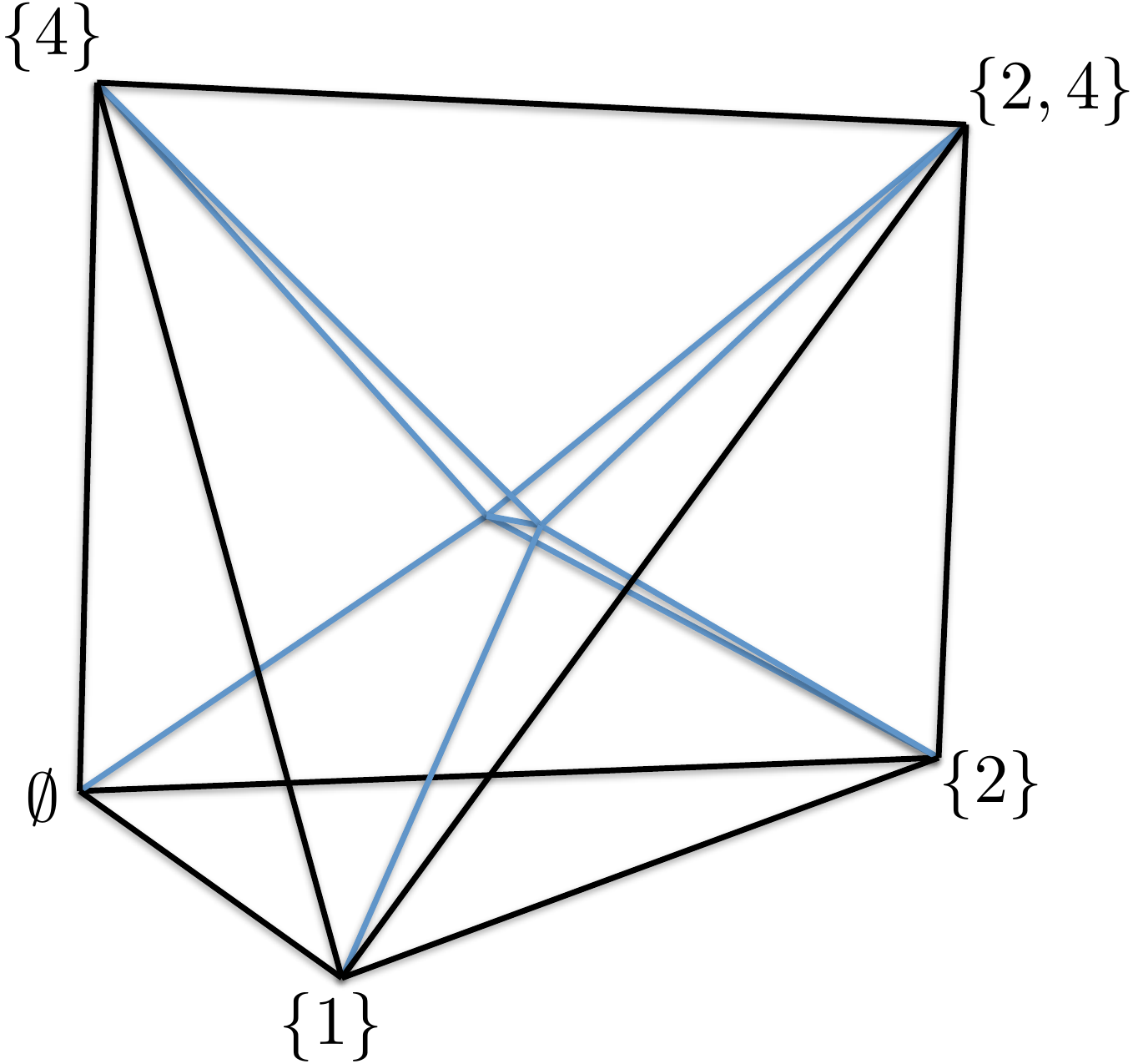}
\caption{A Schlegel diagram of $P_{C_4\rightarrow \downspoon}$ with the edges of the facet projected on in black. The left interior vertex is $\{3\}$ and the right interior vertex is $\{1,3\}$.  This polytope is described in Example~\ref{ex:polyC4Ind}.} \label{fig:Schlegel}
\end{figure}

\section{Algebras with a straightening law}\label{sec:ASL}

Algebras with a straightening law were introduced and studied in for example \cite{DeConciniEisenbudProcesi1982, hibiWatanabe1985, hibi1987}.
This is the basic setup: As before let $\mathbf{k}$ be a field. Let $R$ be a  $\mathbf{k}$-algebra with a generating subset $D$, and assume that there is a poset structure on $D$. A monomial is a product $\alpha_1\alpha_2\cdots\alpha_p$ where $\alpha_i\in D$, and it is \emph{standard} if  $\alpha_1\le\alpha_2\le\cdots\le\alpha_p$ in the poset.
The ring $R$ is \emph{an algebra with straightening laws on $D$} if
\begin{itemize}
\item[(1)] The set of standard monomials is a basis of the algebra $R$ as a vector space over $k$, and
\item[(2)] if $\alpha$ and $\beta$ in $D$ are incomparable and $\alpha\beta=\sum{r_i\gamma_{i1}\gamma_{i2}\cdots\gamma_{ip_i}},$
where $0\neq r_i\in \mathbf{k}$ and $\sum{r_i\gamma_{i1}\gamma_{i2}\cdots\gamma_{ip_i}}$ is a linear combination of standard monomials, then $\gamma_{i1}\le\alpha,\beta$ for every $i$.
\end{itemize}
Hibi~\cite{hibi1987} considered the case when $D$ is a distributive lattice and $R$ is defined to  satisfy the relations $\alpha\beta=(\alpha\vee\beta)(\alpha\wedge\beta)$, for $\alpha,\beta\in D$. He proved that $R$ is an integral domain, an algebra with straightening laws, normal semigroup and Cohen-Macaulay. This can be translated into the language of ideals of graph homomorphisms.

First we recall some basic poset theory, for example from the textbook \cite{stanley1998}. A lower ideal $L$ of a poset $P$ is a subset of $P$ satisfying that if $p\leq q\in L$ then $p\in L$. The lower ideals are ordered by inclusion in a poset $\mathcal{J}(P)$. According to Birkhoff's theorem any distributive poset $D$ is isomorphic to a poset of lower ideals $\mathcal{J}(P)$.

Now we relate to Hibi's results. For any distributive lattice $D$ isomorphic to $\mathcal{J}(P)$, define the bipartite graph $B_P$ with vertex set $P\times \{l,u\}$ and edges $(p,l)-(q,u)$ whenever $p \geq q$ in $P$.

\begin{lemma}\label{lemma:hibify}
There is a bijection $\xi$ from the set of lower ideals of $P$ to the set of maximal independent sets of $B_P$ by $\xi(L) = L\times \{l\} \cup (P \setminus L) \times \{u \}$.
\end{lemma}
\begin{proof}
First note that the maximal independent sets of $B_P$ have the same cardinality as $P$, since $P\times \{l\}$ is independent in $B_P$, and every
edge $(p,l)-(p,u)$ is present.

Any set $\xi(L)$ of vertices in $B_P$ is at least of the right cardinality. Let $L$ be a lower ideal of $P$ and assume that $\xi(L)$ is not independent in $B_P$. Then there is an edge $(p,l)-(q,u)$ where $p \geq q$, and $ p \in L$ and $q \not \in L$. This is a contradiction.

Now let $L_1\times \{l\} \cup L_2 \times \{u \}$ be a maximal independent set of $B_P$. Because of the edges $(p,l)-(p,u)$ and independence we have that $L_1 \cap L_2 = \emptyset$, and because of maximality that $L_2 = P \setminus L_1.$ We conclude that $\xi(L_1)=L_1\times \{l\} \cup L_2 \times \{u \}$.
\end{proof}

\begin{theorem}
Let $D$ be a distributive lattice isomorphic to $\mathcal{J}(P)$. Then Hibi's algebra is isomorphic to
$R_{B_P\rightarrow \downspoon}^{\mathtt{top}} / I_{B_P\rightarrow \downspoon}^{\mathtt{top}}.$
\end{theorem}
\begin{proof}
Using Birkhoff's theorem we can describe Hibi's algebra as
\[ \mathbf{k}[ r_{L} : L \in \mathcal{J}(P) ] / \langle r_{L_1}r_{L_2} - 
r_{L_1 \cup L_2}r_{L_2 \cap L_2} : L_1,L_2 \in \mathcal{J}(P) \rangle.
\]
According to Theorem~\ref{thm:bipartiteGrobner} characterizing the minimal basis of $I_{B_P\rightarrow \downspoon}$, we have 
that 
\[  R_{B_P\rightarrow \downspoon}^{\mathtt{top}} = 
\mathbf{k}[ r_{S} : S \textrm{ maximal independent in }B_P]
\]
and $I_{B_P\rightarrow \downspoon}^{\mathtt{top}} $ is
\[
\left \langle
\begin{array}{c}
r_{L_1\times \{l\} \cup (P \setminus L_1) \times  \{u\} }  
r_{L_2\times \{l\} \cup (P \setminus L_2) \times  \{u\} }  
 -  \\
r_{(L_1\cup L_2 ) \times \{l\} \cup (P \setminus (L_1 \cup L_2)) \times  \{u\} }  
r_{(L_1\cap L_2 ) \times \{l\} \cup (P \setminus (L_1 \cap L_2)) \times  \{u\} }  
\end{array}
 : L_1,L_2 \in \mathcal{J}(P) 
 \right \rangle
\]
since $(P\setminus L_1) \cap (P\setminus L_2) =  P \setminus (L_1 \cup L_2)$ and $(P\setminus L_1) \cup (P\setminus L_2) =  P \setminus (L_1 \cap L_2)$. Now by the bijection $\xi$ of Lemma~\ref{lemma:hibify} we are done.
\end{proof}

Starting from a distributive poset Hibi defined a binomial ideal using square-free quadratic relations, while we define toric varieties from graph homomorphisms and then prove that the corresponding ideals under certain conditions is generated by square-free quadratic binomials. One can also prove that Hibi's binomial ideal is the kernel of the homomorphism sending indeterminates corresponding to lower ideals to indeterminates corresponding to their elements (made homogenous). Using that result one can realize the ideals of graph homomorphisms from bipartite graphs to $\downspoon$ as kernels of the map from the ring whose indeterminates are the lower ideals in the poset gotten by tilting the bipartite graph horizontally and then complementing the upper part. 

This in a sense tells us that from a toric geometry point of view the study of lower ideals in posets and independent sets in bipartite graphs are almost the same, but there is a richer algebraic structure on the bipartite side since the toric ideals of distributive posets only contain the top-graded information. It would be very interesting to understand for concrete applications of Hibi's poset ideals what the not top-graded part is.

A similar connection from distributive lattices to bipartite graphs, but regarding monomial ideals and Rees algebras, was done by Herzog and Hibi \cite{HH}. One good question is if Rees algebras of monomial ideals associated to ideals of graph homomorphisms could be understood with their methods.

\section{Graph coloring} \label{sec:Colorings}

One of the main objectives of graph theory is to determine the \emph{chromatic number $\chi(G)$} of a graph $G$. This is the smallest number such that for any $n\geq \chi(G)$ there is a graph homomorphism from $G$ to $K_n$. The difficult part is usually not to find colorings, but to obstruct them, providing lower bounds for $\chi(G)$. In the turning this question into algebra with a functorial perspective, it's not uncommon to use test graphs \cite{BabsonKozlov2007}. 

This is the general setup: Say that there would exist a graph homomorphism $G \rightarrow K_n$, and that $\mathbf{A}(G,H)$ is the image under a functor into an algebraic category of the graph homomorphism $G\rightarrow H$. Then for any \emph{test graph} $T$ there would be a morphism from $\mathbf{A}(T,G)$ to $\mathbf{A}(T,K_n)$. Now, the game is that the test graph $T$ should be simple enough to calculate  $\mathbf{A}(T,K_n)$ explicitly, and then by some algebraic obstruction theory, the non-existence of a morphism from $\mathbf{A}(T,G)$ to $\mathbf{A}(T,K_n)$ would imply that there is no graph homomorphism from $G$ to $K_n$ and $\chi(G)> n$.

Applying this idea to ideals of graph homomorphism we need a test graph $T$ with $I_{T \rightarrow K_n}$ explicitly described. In the last example of Section 5 
we noted that the ideal $I_{K_3 \rightarrow K_4}$ is generated by the degree 12 binomial
 \[ 
 \begin{array}{l}
r_{123}r_{214}r_{341}r_{432}r_{231}r_{142}r_{413}r_{324}r_{312}r_{421}r_{134}r_{243} - \\
r_{124}r_{213}r_{342}r_{431}r_{234}r_{143}r_{412}r_{321}r_{314}r_{423}r_{132}r_{241}.
\end{array}
 \]
This shows that $K_3$ is a suitable test graph for four-colorings.
\begin{proposition}\label{theo:graphColoring}
Let $\xi$ be a four-coloring of $G$, that is, a graph homomorphism from $G$ to $K_4$. And let
\[ r_{\phi_1}r_{\phi_2} \cdots r_{\phi_d} -  r_{\phi_1'}r_{\phi_2'} \cdots r_{\phi_d'} \]
be a binomial without common variables in $I_{K_3 \rightarrow G}$. If $d<12$  then there exists
a permutation $\pi \in \mathcal{S}_d$ such that $\xi \circ \phi_i = \xi \circ \phi_{\pi(i)}'$.
\end{proposition}
\begin{proof}
The map $\xi$ induces a homomorphism from $I_{K_3 \rightarrow G}$ to $I_{K_3 \rightarrow K_4}$ by sending $r_\phi$ to $r_{\xi \circ \phi}$.
The ideal $I_{K_3 \rightarrow K_4}$ is generated by one binomial of degree 12, so anything of smaller degree is sent to zero, and this can only be achieved by identifications of variables.
\end{proof}
The following example is included since it's a baby version of the equivariant method employed by Lov\'asz~\cite{LOV} in his proof of the Kneser conjecture. For a contemporary view of Lov\'asz proof we refer to Babson and Kozlov~\cite{BabsonKozlov2007}. We hope that our method can be extended to find the chromatic number of graphs with huge symmetries, where the obstruction to coloring is not local.

\begin{example}
Any four coloring of the one-skeleton of the octahedron has two antipodal vertices of the same color:
Label the vertices of the octahedron graph $O$ as in Figure~\ref{fig:octahedron}. The binomial
$r_{135}r_{146}r_{236}r_{245}-r_{136}r_{145}r_{235}r_{246}$
is in $I_{K_3 \rightarrow O}$. This binomial is of a degree less than 12, and Proposition~\ref{theo:graphColoring} applies. Let's focus on how the permutation $\pi$ will permute $i=1$. There are four different options, and they are in Table~\ref{tab:Ocolorings}. For every value of $\pi(1)$ there is an identification $\boldsymbol{\xi(u)=\xi(v)}$ with $u$ and $v$ antipodal.
\end{example}
\begin{figure}
\center
\includegraphics[scale=0.5]{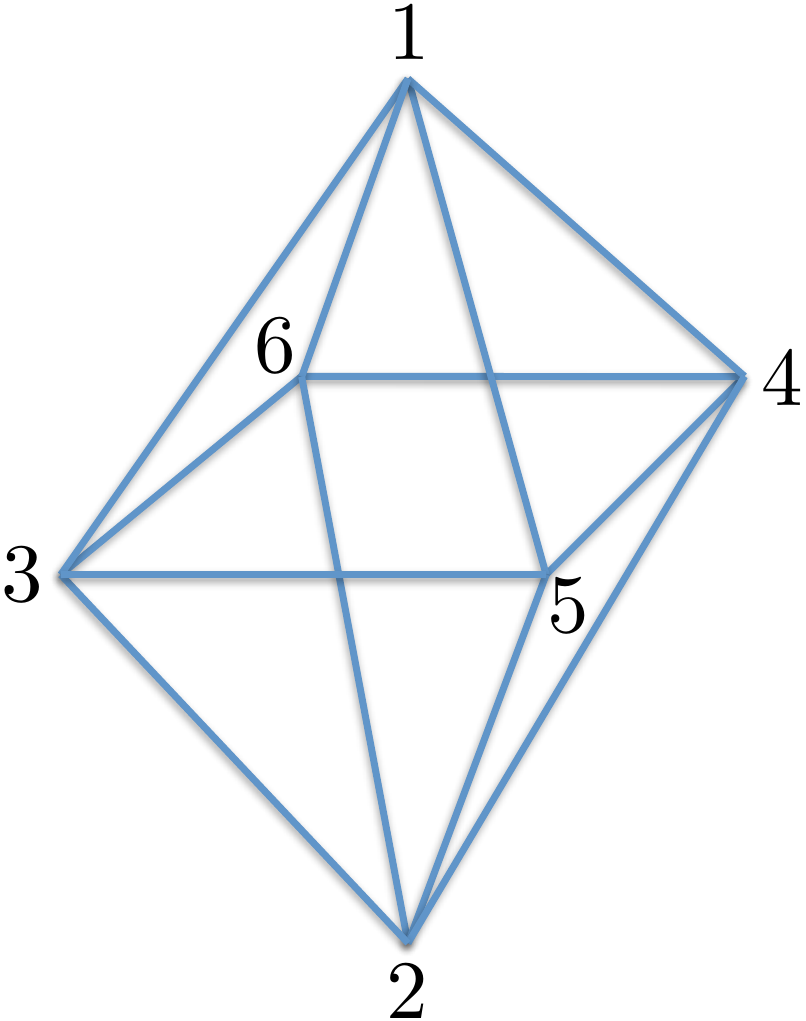}
\caption{The octahedron graph $O$ with labels as in the first example of Section 11.} \label{fig:octahedron}
\end{figure}
\begin{table}
\begin{tabular}{c|ccc}
$\pi(1)$ & \multicolumn{3}{c}{The identifications} \\
\hline
1 & $\xi(1)=\xi(1)$ & $\xi(3)=\xi(3)$ & $\boldsymbol{\xi(5)=\xi(6)}$ \\
2 & $\xi(1)=\xi(1)$ & $\boldsymbol{\xi(3)=\xi(4)}$ & $\xi(5)=\xi(5)$ \\
3 & $\boldsymbol{\xi(1)=\xi(2)}$ & $\xi(3)=\xi(3)$ & $\xi(5)=\xi(5)$ \\
4 & $\boldsymbol{\xi(1)=\xi(2)}$ & $\boldsymbol{\xi(3)=\xi(4)}$ & $\boldsymbol{\xi(5)=\xi(6)}$ \\
\end{tabular}
\caption{The identifications forced by different permutations $\pi$, with antipodal ones are in bold.}\label{tab:Ocolorings}
\end{table}
The smallest graph of chromatic number larger than four is $K_5$, and it is combinatorially trivial to establish this. With this example we demonstrate that the method doesn't break down for the first case.
\begin{example}
The graph $K_5$ is not four-colorable:
Assume to the contrary\linebreak that $K_5$ is four-colorable by a homomorphism $\xi:K_5 \rightarrow K_4$. The binomial\linebreak
$ r_{123}r_{145}r_{325}r_{341}r_{521}r_{543} - r_{125}r_{143}r_{321}r_{345}r_{523}r_{541} $
is in the ideal $I_{K_3 \rightarrow K_5}$. Let $\pi$ be the permutation promised by Proposition~\ref{theo:graphColoring}.
 In Table~\ref{tab:K5colorings} are the forced identifications of colorings for different values of $\pi(1)$. For every value of $\pi(1)$ there is an identification $\boldsymbol{\xi(u)=\xi(v)}$ with $u$ and $v$ adjacent, contradicting that $\xi$ is a graph homomorphism from $K_5$ to $K_4$.
\end{example}
\begin{table}
\begin{tabular}{c|ccc}
$\pi(1)$ & \multicolumn{3}{c}{The identifications} \\
\hline
1 & $\xi(1)=\xi(1)$ & $\xi(2)=\xi(2)$ & $\boldsymbol{\xi(3)=\xi(5)}$ \\
2 & $\xi(1)=\xi(1)$ & $\boldsymbol{\xi(2)=\xi(4)}$ & $\xi(3)=\xi(3)$ \\
3 & $\boldsymbol{\xi(1)=\xi(3)}$ & $\xi(2)=\xi(2)$ & $\boldsymbol{\xi(3)=\xi(1)}$ \\
4 & $\boldsymbol{\xi(1)=\xi(3)}$ & $\boldsymbol{\xi(2)=\xi(4)}$ & $\boldsymbol{\xi(3)=\xi(5)}$ \\
5 & $\boldsymbol{\xi(1)=\xi(5)}$ & $\xi(2)=\xi(2)$ & $\xi(3)=\xi(3)$ \\
6 & $\boldsymbol{\xi(1)=\xi(5)}$ & $\boldsymbol{\xi(2)=\xi(4)}$ & $\boldsymbol{\xi(3)=\xi(1)}$ \\
\end{tabular}
\caption{The identifications forced by different permutations $\pi$, with adjacent ones in bold.}\label{tab:K5colorings}
\end{table}
One interesting aspect of this method, is that given the low-degree binomial it is elementary to check that the proof of $\chi(G)>4$ is correct. In a sense, the binomial is a certificate that can be easily tested and communicated. There are very efficient methods to find certificates for huge pseudo-random graphs where the obstruction is local using other algebraic methods \cite{LOERAetal}. Software like 4ti2 \cite{42} is efficient to find generating sets, but there is nothing off the shelf that only gives binomials up to a certain degree without doing unnecessary calculations. The development of software for finding low degree binomials fast, would enable large scale tests of our method.

\end{document}